\documentclass[reqno,12pt]{amsart}
\usepackage{amsmath}
\usepackage{amssymb}
\usepackage{amscd}
\usepackage{graphicx}
\usepackage{color}

\usepackage{comment}

\usepackage{amsfonts}

\usepackage{mathptmx}
\usepackage[colorlinks=true,citecolor=blue]{hyperref}
\hypersetup{
	pdftitle={An affirmative answer to Owings's sumset question},
	pdfauthor={Wen Huang; Zhengxing Lian; Song Shao; Rongzhong Xiao; Leiye Xu; Shuhao Zhang}
}

\usepackage[nameinlink,capitalise,noabbrev]{cleveref}

\usepackage{tikz}
\usetikzlibrary{decorations.pathmorphing}

\allowdisplaybreaks[4]
\makeatletter
\@namedef{subjclassname@2020}{%
	\textup{2020} Mathematics Subject Classification}
\makeatother

\usepackage[top=25mm, bottom=27mm, left=27mm, right=27mm]{geometry}
\newtheorem{thm}{Theorem}[section]
\newtheorem{lemma}[thm]{Lemma}

\newtheorem{cl}{Claim}
\newtheorem{prop}[thm]{Proposition}

\newtheorem{ques}[thm]{Question}

\newtheorem{defn}[thm]{Definition}
\newtheorem{rem}[thm]{Remark}

\crefname{thm}{Theorem}{Theorems}
\Crefname{thm}{Theorem}{Theorems}
\crefname{lemma}{Lemma}{Lemmas}
\Crefname{lemma}{Lemma}{Lemmas}
\crefname{prop}{Proposition}{Propositions}
\Crefname{prop}{Proposition}{Propositions}
\crefname{cor}{Corollary}{Corollaries}
\Crefname{cor}{Corollary}{Corollaries}
\crefname{defn}{Definition}{Definitions}
\Crefname{defn}{Definition}{Definitions}
\crefname{ques}{Question}{Questions}
\Crefname{ques}{Question}{Questions}
\crefname{con}{Conjecture}{Conjectures}
\Crefname{con}{Conjecture}{Conjectures}
\crefname{rem}{Remark}{Remarks}
\Crefname{rem}{Remark}{Remarks}
\crefname{ex}{Example}{Examples}
\Crefname{ex}{Example}{Examples}
\crefname{cl}{Claim}{Claims}
\Crefname{cl}{Claim}{Claims}
\crefname{fact}{Fact}{Facts}
\Crefname{fact}{Fact}{Facts}
\crefname{step}{Step}{Steps}
\Crefname{step}{Step}{Steps}
\crefname{section}{Section}{Sections}
\Crefname{section}{Section}{Sections}
\crefname{equation}{Equation}{Equations}
\Crefname{equation}{Equation}{Equations}

\newcommand{\Nzero}{\mathbb N_0}
\newcommand{\bN}{\beta\Nzero}
\newcommand{\Nstar}{\Nzero^{*}}
\newcommand{\two}{\{0,1\}}

\newcommand{\Torus}{\mathbb T}

\def \N {\mathbb N}

\def \Z {\mathbb Z}

\numberwithin{equation}{section}

\begin{document}

\title[]{An affirmative answer to Owings's sumset question}

	\author[]{Wen Huang, Zhengxing Lian, Song Shao, Rongzhong Xiao,\\ Leiye Xu, and Shuhao Zhang}	
	\address[Wen Huang, Song Shao, Rongzhong Xiao, Leiye Xu, and Shuhao Zhang]{School of Mathematical Sciences, University of Science and Technology of China, Hefei, Anhui, 230026, PR China}
	\email{wenh@mail.ustc.edu.cn}
	\email{songshao@ustc.edu.cn}
	\email{xiaorz@mail.ustc.edu.cn}
	\email{leoasa@mail.ustc.edu.cn}
	\email{yichen12@mail.ustc.edu.cn}

   \address[Zhengxing Lian]{School of Mathematical Sciences, Xiamen University, Xiamen, Fujian, 361005, PR China}
   \email{lianzx@xmu.edu.cn}
%
%
	\thanks{This work is supported by the National Key R\&D Program of China (No. 2024YFA1013601, 2024YFA1013600) and the National Natural Science Foundation of China (No. 123B2007, 12426201, 12371196).
}
	
\begin{abstract}
We give an affirmative answer to Owings's sumset question: for any $2$-coloring of natural numbers, there is an infinite $B\subseteq\N$ such that $B+B$ is monochromatic. More generally, for every $m,\ell\in\N$ and every $2$-coloring of $\N$, there is an infinite $B\subseteq\N$ such that
$$
(m+\ell)B\cup\{mx+\ell y:x,y\in B,\ x<y\}
$$
is monochromatic. 

\end{abstract}

\maketitle

\section{Introduction}
One of the central themes in combinatorial number theory is the search for
structured configurations inside large sets of natural numbers.
In this paper, we focus on the unrestricted infinite sumsets in finite colorings. For $r\in \N$,  an {\em $r$-coloring} of the elements of a set $S$ is a mapping $a: S\rightarrow T$, where $|T|=r$, and typically, $T=\{1,2,\ldots, r\}$.
A set $M\subseteq S$ is {\em monochromatic} (for $a$) if $a|_M$ is constant.

\subsection{Owings's question}

In 1974, Hindman \cite[p.~1]{Hindman74} proved a conjecture of Graham and Rothschild by showing the following: For any finite coloring of natural numbers, there is a sequence $\{x_n\}_{n\in \N}$ in $\N$ such that the set
$$\left\{\sum_{n\in \alpha}x_n:\emptyset \neq \alpha\subseteq \N\ \text{finite}\right\}$$ is monochromatic.
It is important that no term $x_n$ is used more than once in the finite sum. Hindman's theorem is false if one allows so much as a single repetition (see \cite[p.~138]{Hindman79} or \cite[p.~19]{Neil79} for a counterexample). In \cite[pp.~19--20]{Neil79}, Hindman studied the following question: If $r\in \N$ and $\N=\bigsqcup_{i=1}^rC_i$, must there exist $i\in \{1,2,\ldots, r\}$ and an infinite subset $A$ such that $A+A:=\{x+y: x,y\in A\}\subseteq C_i$ ? The question was asked for $r=2$ by Owings:

\medskip
\noindent {\bf Owings's question}. $($\cite[Problem E2494]{Owings74}$)${\bf .} {\em Prove or disprove: Given any subset $B$ of $\N$, there exists an infinite set $A\subseteq \N$ such that $A+A\subseteq B$ or $A+A\subseteq \N\setminus B$.}

Hindman~\cite[Definition~2.2]{Neil79}
introduced admissible partitions, requiring one cell to contain arbitrarily
long arithmetic progressions of even integers with a fixed common
difference.  He proved that every admissible two-cell partition contains
$B+B$ for some infinite $B$
\cite[Corollary~2.10]{Neil79}, and constructed an admissible
three-cell partition with no monochromatic $B+B$
\cite[Theorem~2.4]{Neil79}.
The non-admissible two-color case
remained open, as recorded by Hindman and
Strauss~\cite[p.~458]{HS}.



A set $A=\{a_1<a_2<\cdots\}\subseteq \N$ is {\em syndetic} if there exists $k\in \N$ such that for each $n\in\N$, $a_{n+1}-a_n\le k$, and  $A\subseteq \N$ is {\em thick} if it contains arbitrarily long
intervals.
In \cite[Proposition 1.5]{KousekRadic24}, Kousek and Radi{\'c} constructed a $3$-coloring, where each cell is syndetic and for any infinite $B\subseteq \N$, $B+B$ is not monochromatic. They also observed that Owings's question is
equivalent to its shifted form: every two-coloring has a monochromatic
$B+B+t$ for some infinite $B\subseteq\N$ and $t\ge0$ (see \cite[Remark 6.2]{KousekRadic24}). Note that Banakh and Zdomskyy recorded
a semifilter/unsplittability reformulation \cite[Section~28]{BanakhZdomskyy2006}, while
Hindman's survey on infinite partition-regular matrices retained Owings's question as a
basic unresolved problem \cite[p.~211]{Hindman2018Matrices}.

Interest in the problem was renewed through extensions to other additive
structures and to questions involving infinite cardinals.   Hindman, Leader, and Strauss
studied rational vector spaces and the reals, showing sensitivity to dimension
and cardinal arithmetic \cite{HindmanLeaderStrauss2017}.  Komj\'ath, Leader,
Russell, Shelah, Soukup, and Vidny\'anszky exposed substantial set-theoretic
features in the real case \cite{KomjathEtAl2019}; Leader and Russell obtained
positive results in sufficiently large rational dimension
\cite{LeaderRussell2020}.  Fern\'andez-Bret\'on, Sarmiento Rosales, and Vera
developed Owings-type relations for general abelian groups and varying colour
and target cardinalities \cite{FernandezEtAl2024}, and Leader and Williams
treated countable colourings of abelian groups \cite{LeaderWilliams2024}.
Guzm\'an-Vega, Fern\'andez-Bret\'on, and Sarmiento Rosales subsequently examined
which Hindman- and Owings-type statements survive without the Axiom of Choice
\cite{GuzmanEtAl2026}. Owings's question also entered combinatorics on words through work of Wojcik and Zamboni on
monochromatic factorizations \cite{WojcikZamboni2018Periodicity,
WojcikZamboni2018Words}; Wojcik's thesis records the same connection
\cite{Wojcik2019Thesis}.

In this paper, we answer Owings's question affirmatively.
\begin{thm}\label{main1}
Let $\N=C_1\sqcup C_2$. Then there exist $i\in \{1,2\}$ and an infinite $B\subseteq \N$ such that
$$B+B\subseteq C_i.$$
\end{thm}

Together with Hindman's
three-cell counterexample, \cref{main1} therefore gives the exact
finite-color threshold: the assertion holds for two colors and fails for
every finite number of colors at least three.

One may also ask about the $3$-fold version of Owings's question. That is, is it true that for any subset $B$ of natural numbers, there exists an infinite set $A\subseteq \N$ such that $A+A+A\subseteq B$ or $A+A+A\subseteq \N\backslash B$ ? In \cref{sec6-3}, we construct a $2$-coloring of natural numbers such that the $3$-fold version of Owings's question fails.

\subsection{On sumsets structures in sets with positive density}
While attempting to formulate a conjecture which would be in the same relation to Hindman's theorem as Szemer\'edi's theorem \cite[p.~199]{Szemeredi75} (that every positive upper density subset of natural numbers contains arbitrarily long arithmetic progressions) is to van der Waerden's theorem \cite[pp.~212--216]{Waerden27} (that one of the cells of each finite partition of natural numbers contains arbitrarily long arithmetic progressions), Erd\H{o}s proposed the following conjecture:

\medskip
\noindent {\bf Erd\H{o}s Conjecture}. $($\cite[p.~305]{Erdos75}$)${\bf .} \ {\em For any $A\subseteq \N$ with positive upper density, there is an infinite $B\subseteq A$ and an integer shift $t$ such that}
	$$B\oplus B:=\{b_1+b_2:b_1,b_2\in B,b_1\neq b_2\}\subseteq A-t.$$

Note that $B+B=(B\mathbin{\oplus}B)\cup2B$. This conjecture was recently resolved by Kra, Moreira, Richter, and Robertson \cite[Theorem~1.2]{KraMoreiraRichterRobertson24}.
Note that in the formulation of Erd\H{o}s Conjecture, it is not possible to omit the shift by $t$ or remove the condition $b_1\neq b_2$ in the conclusion (see, for example, \cite[Examples 2.3, 3.6]{KraMoreiraRichterRobertson25-1}).
Kra, Moreira, Richter, and Robertson placed Owings's question among a larger
family of problems about infinite sumset configurations; their
\cite[Question~3.8]{KraMoreiraRichterRobertson25-1} is precisely Owings's
question.  


Kousek and Radi{\'c} studied the form $B+B$ in sets with large density in \cite{KousekRadic24}. They showed that if $A\subseteq \N$ satisfies $\overline{d}(A)>5/6$ or $\underline{d}(A)>3/4$\footnote{A set $A\subseteq \N$ has {\em upper density} given by $\overline{d}(A)=\limsup_{N\to\infty}{|A\cap [N]|}/{N}$ and {\em lower density} given by $\underline{d}(A)=\liminf_{N\to\infty}{|A\cap [N]|}/{N}$, where $[N]=\{1,2,\ldots, N\}$.}, then there is an infinite set $B\subseteq \N$ such that $B+B\subseteq A$ \cite[Theorem~1.2]{KousekRadic24}. Also, they showed that if the answer to Owings's question turns out to be negative, then the sum of upper density and lower density of each color is $1$. That is, for $\N=C_1\sqcup C_2$, if there do not exist infinite $B\subseteq \N$, a shift $t\in \N$ and $i\in \{1,2\}$ such that $B+B+t\subseteq C_i$, then $\overline{d}(C_i)+\underline{d}(C_i)=1$ for both $i=1,2$ \cite[Proposition~6.3]{KousekRadic24}.

Kousek generalized Kra, Moreira, Richter, and Robertson's result to the
following one, thereby verifying
\cite[Conjecture~3.10]{KraMoreiraRichterRobertson25-1}. Recall that the
\emph{upper Banach density} of $A\subseteq\N$ is
\[
 d^*(A)=\limsup_{N\to\infty}\sup_{M\geq0}
 \frac{|A\cap\{M+1,\ldots,M+N\}|}{N}.
\]
\begin{thm}\label{thm:Kousek-shifted-weighted}
	$($\cite[Theorem~1.2]{Kousek25}$)${\bf .} For any $A\subseteq \N$ with positive upper Banach density and $\ell, m\in \N$, there are an infinite $B\subseteq \N$ and a shift $t\ge 0$ such that
	$$\{m x+\ell y:x,y\in B,x< y\}\subseteq A-t.$$
\end{thm}

Also, he gave an unrestricted version of the above result:

\begin{thm}\label{thm1-1}
	$($\cite[Theorems~1.6 and 1.7]{Kousek25}$)${\bf .} Let $\ell,m\in \N$ and $k=m/\ell$. For any $A\subseteq \N$ with $\underline{d}(A)>1/2$ or $\overline{d}(A)>1-\frac{1}{k+2}$, there are an infinite $B\subseteq \N$ and a shift $t\ge 0$ such that
	$$\{mx+\ell y:x,y\in B,x\le y\}\subseteq A-t.$$
\end{thm}

\subsection{Weighted form}
Based on \cref{thm1-1}, it is natural to ask the following weighted restricted version of Owings's question:

\begin{ques}\label{q2}
	Let $m,\ell\in \N$. Is it true that for any $2$-coloring of natural numbers, there exists an infinite set $B\subseteq \N$ such that
	\begin{equation}\label{eq:weighted-target}
	(m+\ell)B\cup\{mx+\ell y:x,y\in B,\ x<y\}
	\end{equation}
	is monochromatic?
\end{ques}

\begin{rem}\label{rem1.6}
  When $m\neq \ell$, we have that
  $\{mx+\ell y:x,y\in A,x<y\}\cup \{mx+\ell y:x,y\in A,x>y\}=\{mx+\ell y:x,y\in A,x\neq y\}$.
  Hindman \cite[Theorem~2.11]{Neil79} showed: Let $m\ge 2$. Then there exists a $2$-coloring $\N=C_1\sqcup C_2$ such that there do not exist an infinite subset $B\subseteq \N$ and $i\in \{1,2\}$ with $\{mx+y: x,y\in B,x\neq y\}\subseteq C_i$. Thus we cannot replace $\{mx+\ell y:x,y\in A,x<y\}$ by $\{mx+\ell y:x,y\in A,x\neq y\}$ in \cref{q2}.


\end{rem}

We answer \cref{q2} affirmatively for every ordered pair of positive
coefficients.  The order condition is essential: $m$ is attached to the
earlier element and $\ell$ to the later one, so the ordered pairs $(m,\ell)$
and $(\ell,m)$ are genuinely different.

\begin{thm}\label{main2}
	Fix $m,\ell\in\N$. Let $\N=C_1\sqcup C_2$. Then there exist
	$i\in\{1,2\}$ and an infinite $B\subseteq\N$ such that
	\[
	(m+\ell)B\cup\{mx+\ell y:x,y\in B,\ x<y\}\subseteq C_i.
	\]
\end{thm}

\begin{rem}
	The restriction to two colors in \cref{main2} is necessary.  See
	\cref{prop6-1} for a concrete three-coloring obstruction.
\end{rem}

In
\cref{sec6}, we present counterexamples mentioned above. Also
\Cref{prop6-2} shows that for any $\ell,m\in \N$ with $\ell\neq m$, there is a $2$-coloring of
$\N$ such that, for every infinite $B\subset\N$ and all
$t_1,t_2\in\mathbb Z$, the set
\[
\{mx+\ell y+t_1:x,y\in B,\ x<y\}
\cup
\{mx+\ell y+t_2:x,y\in B,\ x>y\},
\]
whenever contained in $\N$, is not monochromatic.

\subsection*{Organization of the paper}
In \cref{sec2}, we collect the combinatorial, topological-dynamical, and
ultrafilter preliminaries used throughout the paper. Although
\cref{main1} is the case $m=\ell=1$ of \cref{main2}, in \cref{sec3} we
give a shorter independent proof of \cref{main1}. In \cref{sec4}, we
prove the stronger weighted theorem by a different refinement of the same
general affine-ultrafilter strategy. In \cref{sec6}, we present
counterexamples, and in \cref{appA}, we prove \cref{thm2-1}.

\section{Preliminaries}\label{sec2}
In this section, we introduce some notations, notions and results used in the paper. Throughout the paper, $\N=\{1,2,\ldots\}$,
$\Nzero=\{0,1,2,\ldots\}$, and  $\Z$ is the set of integers.
All combinatorial sets and largeness notions are considered in $\N$ (or
in $\Nzero$ when this is explicitly stated). The set $\Z$ is used only
as auxiliary algebraic and dynamical notation, for example for two-sided
iterates and coordinates associated with homeomorphisms and for integer
translation parameters; it is not the ambient set for the combinatorial
largeness notions below. Unless otherwise stated, if $n<m$ belong to
$\Nzero$, then
$[n,m]=\{n,n+1,\ldots,m\}$ and
$(n,m)=\{n+1,n+2,\ldots,m-1\}$.

\subsection{Topological dynamical systems}
A {\em topological dynamical system} (simply referred to as \emph{a system}) is a pair $(X,T)$, where $X$ is a
compact metric space with a metric $d_X$ and $T:X\to X$ is a homeomorphism.  A nonempty closed set
$Y\subseteq X$ is a \emph{subsystem} if $T(Y)=Y$.  It is
\emph{minimal} if it contains no proper nonempty subsystem.  Equivalently,
the forward orbit $\{T^ny:n\in\Nzero\}$ is dense in $Y$ for every $y\in Y$.
By Zorn's lemma, every nonempty subsystem contains a minimal subsystem.
For $x\in X$, put
\begin{equation}\label{eq:omega-def}
\omega(x)=\bigcap_{N\geq0}
\overline{\{T^nx:n\geq N\}}.
\end{equation}
Since $T$ is a homeomorphism, both $T\omega(x)\subseteq\omega(x)$ and
$T^{-1}\omega(x)\subseteq\omega(x)$ hold. Hence
$T\omega(x)=\omega(x)$, and $(\omega(x),T)$ is a subsystem of $(X,T)$.

A set that is both open and closed is called \emph{clopen}.  A space is
\emph{zero-dimensional} if it has a base consisting of clopen sets.

A {\em factor map} $\pi: X\rightarrow Y$ between two systems $(X,T)$
and $(Y, S)$ is a continuous surjective map which intertwines the
actions (i.e. $\pi\circ T= S\circ \pi$); one says that $(Y, S)$ is a {\it factor} of $(X,T)$ and
that $(X,T)$ is an {\it extension} of $(Y,S)$. The systems are said to be {\em isomorphic} if $\pi$ is bijective.

The following two lemmas are classical.

\begin{lemma}\cite[Theorem~3.1]{Ye92}\label{lem:power-components}
Let $(X,T)$ be minimal, and let $k\geq1$. Every
$T^k$-minimal subset is clopen. The distinct $T^k$-minimal subsets form a
finite partition of $X$, and $T$ permutes them transitively.
\end{lemma}


\begin{lemma}\cite[Theorem~1.15]{Auslander88}\label{lem:semiopen}
Every factor map $\pi:(X,T)\to(Y,S)$ between minimal systems is semi-open: if $U\subset X$ is nonempty and open, then
$\pi(U)$ has nonempty interior.
\end{lemma}

\subsection{Equicontinuity and maximal equicontinuous factors}

Let $(X,T)$ be a system. It is
\emph{equicontinuous} if the family $\{T^n:n\in\Z\}$ is equicontinuous;
equivalently, for every $\varepsilon>0$ there is $\delta>0$ such that
\[
 d_X(x,x')<\delta
 \quad\Longrightarrow\quad
 d_X(T^nx,T^nx')<\varepsilon, \quad \forall n\in\Z.
\]
An equicontinuous factor of $(X,T)$ is a factor which is an
equicontinuous system.

A factor map $\pi_{\rm eq}:X\to Z$ is called the \emph{maximal
equicontinuous factor} if $(Z,R)$ is equicontinuous and every
equicontinuous factor $\rho:X\to Y$ factors through $\pi_{\rm eq}$:
there is a unique factor map $\bar\rho:Z\to Y$ such that
\[
 \rho=\bar\rho\circ\pi_{\rm eq}.
\]
The maximal equicontinuous factor exists and is unique up to isomorphism.
If $(X,T)$ is minimal, then $(Z,R)$ is minimal; after choosing an origin,
$Z$ may be represented as a compact monothetic abelian group and
\[
 Rz=z+g,
\]
where $\{ng:n\in\Z\}$ is dense in $Z$.  We always use an equivalent
invariant metric on $Z$.  These facts, as well as the construction by the
equicontinuous structure relation, are standard; see
\cite[Chapter~2, pp.~35--48]{Auslander88}.
For later use, recall that the kernel relation
\[
 \{(x,y)\in X^2:\pi_{\rm eq}(x)=\pi_{\rm eq}(y)\}
\]
is called the \emph{equicontinuous structure relation}.

The following lemma will be used in the proof of \cref{main2}.

\begin{lemma}\label{lem:mef-affine}
Let $(X,T)$ be a minimal system with maximal equicontinuous
factor $\pi:X\to Z$, where $T$ acts on $Z$ by $z\mapsto z+g$.
\begin{enumerate}
\renewcommand{\labelenumi}{\textup{(\roman{enumi})}}
\item If $S:X\to X$ is a homeomorphism commuting with $T$, then there is a
unique $h\in Z$ such that
\[
 \pi(Sx)=\pi(x)+h\qquad(x\in X).
\]
If $S^2=\mathrm{id}$, then $2h=0$.
\item Let $W$ be a compact abelian group and let $f:X\to W$ be continuous.
If $f(Tx)=f(x)+w$ for all $x$, then there are a continuous homomorphism
$A:Z\to W$ and $c\in W$ such that
\[
 f(x)=A\pi(x)+c,
 \qquad
 A(g)=w.
\]
\end{enumerate}
\end{lemma}

\begin{proof}
For (i), the map $\pi\circ S$ is another equicontinuous factor map. By the
universal property of $\pi$, it factors as $\bar S\circ\pi$. Applying the same
argument to $S^{-1}$ shows that $\bar S$ is a homeomorphism. It commutes with
the rotation by $g$. Put $h=\bar S(0)$. On the dense cyclic subgroup
$\{ng:n\in\Z\}$,
\[
 \bar S(ng)=ng+h.
\]
Continuity gives $\bar S(z)=z+h$ for all $z\in Z$. If $S^2=\mathrm{id}$,
then translation by $2h$ is the identity, so $2h=0$.

For (ii), the image system $f(X)$, acted on by translation by $w$, is an
equicontinuous factor of $X$. Hence $f=\bar f\circ\pi$ for a continuous
$\bar f:Z\to W$. Put $c=\bar f(0)$ and $A=\bar f-c$. Then
$A(ng)=nw$ for all $n\in\Z$. If $n_i g\to z$ and $m_i g\to z'$, continuity
gives
\[
 A(z+z')=\lim_i A((n_i+m_i)g)
 =\lim_i(n_i+m_i)w=A(z)+A(z').
\]
Thus $A$ is a continuous homomorphism and $A(g)=w$.
\end{proof}

\begin{lemma}
\label{lem:power-mef}
Let $(X,T)$ be a minimal system with maximal equicontinuous
factor $\pi:X\to Z$, let $q\geq1$, and let $C$ be a
$T^q$-minimal component. Then
\[
 C=\pi^{-1}(\pi(C)),
\]
and
\[
 \pi|_C:(C,T^q)\longrightarrow(\pi(C),R^q)
\]
is the maximal equicontinuous factor of $(C,T^q)$.
\end{lemma}

This lemma is a standard consequence of Ye's cyclic decomposition theorem
\cite[Theorem~3.1]{Ye92} and the invariance of the regionally proximal
relation under powers
\cite[Lemma~2.7, p.~290]{GlasnerHuangShaoWeissYe25}; see also
\cite[Theorem~2.6, pp.~289--290]{GlasnerHuangShaoWeissYe25} for the
characterization of the maximal equicontinuous factor.

\subsection{Return times and minimal systems}

All systems in this subsection are systems in the sense fixed above; in
particular, their phase spaces are compact metric spaces and their
transformations are homeomorphisms. For a minimal system, we use the
notation $\pi:X\to Z$ and $Rz=z+g$ fixed in the preceding subsection.

For $x\in X$ and an open set $U\subset X$, define
\[
 N_T(x,U)=\{n\in\N:T^nx\in U\}.
\]

All the largeness notions used below concern subsets of $\N$. For
$P\subset\N$ and $f\in\Nzero$, put
$P-f=\{n\in\N:n+f\in P\}$. A set $P\subset\N$ is \emph{syndetic} if it
has bounded gaps, and it is \emph{thick} if it contains arbitrarily long
intervals. A set $P\subset\N$ is \emph{piecewise syndetic} if
$\bigcup_{f\in F}(P-f)$ is thick for some finite $F\subset\Nzero$. This
property is preserved by translations within $\N$, finite modifications,
and passing to supersets. A set $S\subset\N$ is \emph{thickly syndetic}
if, for every finite $K\subset\Nzero$, the set
\[
 \{n\in\N:K+n\subset S\}
\]
is syndetic. A thickly syndetic set meets every piecewise syndetic set.
Indeed, if $P$ is piecewise syndetic, choose a finite
$F\subset\Nzero$ such that $\bigcup_{f\in F}(P-f)$ is thick. The
syndetic set $\{n\in\N:F+n\subset S\}$ meets this thick union. If
$n\in P-f$ belongs to the intersection, then $n+f\in P\cap S$.

The following result is an easy observation. We give a proof for completeness.

\begin{lemma}\label{lem:ps-return}
Let $(X,T)$ be a system, let $x\in X$, and let
$M\subset\omega(x)$ be minimal. If $U\subset X$ is open and
$U\cap M\neq\emptyset$, then $N_T(x,U)$ is piecewise syndetic.
\end{lemma}

\begin{proof}
Minimality and compactness give a finite set $F\subset\Nzero$ such that
\[
 M\subset \bigcup_{f\in F}T^{-f}U=:G.
\]
For $L\geq0$, the set $\bigcap_{j=0}^{L}T^{-j}G$ is an open neighborhood of
$M$. Since $M\subset\omega(x)$, arbitrarily large $n$ satisfy
$T^nx\in\bigcap_{j=0}^{L}T^{-j}G$. For each $0\leq j\leq L$, choose
$f_j\in F$ such that $T^{n+j+f_j}x\in U$. It follows that
\[
 [n,n+L]\subset\bigcup_{f\in F}\bigl(N_T(x,U)-f\bigr),
\]
so $N_T(x,U)$ is piecewise syndetic.
\end{proof}

For $s\in\N$, let
\[
 X_{\Delta}^s=\{(x,\ldots,x):x\in X\}\subset X^s.
\]
Write
\[
 T^{(s)}=T\times\cdots\times T:X^s\to X^s
\]
for the diagonal homeomorphism.
A tuple $(x_1,\ldots,x_s)\in X^s$ is called \emph{jointly regionally
proximal} if, for every choice of neighborhoods $G_i$ of $x_i$ and every
open neighborhood $\mathcal D$ of $ X_{\Delta}^s$ in $X^s$, there are
$x_i'\in G_i$ and $k\in\Z$ such that
\[
 (T^kx_1',\ldots,T^kx_s')\in\mathcal D.
\]
We denote the set of such tuples by $Q^{(s)}(X,T)$. The integer iterate is
well defined because $T$ is a homeomorphism. When $s=2$, this is the
usual regionally proximal relation.

The following consequence of Auslander's finite regional proximality
theorem will be used below. See also \cite[Corollary 6.9]{HLY}.

\begin{lemma}\label{lem:Auslander-joint}
Let $(X,T)$ be a minimal system with maximal equicontinuous factor
$\pi:X\to Z$. If $s\in\N$ and
\[
 \pi(x_1)=\cdots=\pi(x_s),
\]
then $(x_1,\ldots,x_s)\in Q^{(s)}(X,T)$.
\end{lemma}

\begin{proof}
The assertion is automatic for $s=1$. Suppose that $s\geq2$. For a
minimal abelian action, the regionally proximal relation is the
equicontinuous structure relation
\cite[p.~327]{AuslanderGroup04}. Thus every pair $(x_1,x_i)$ is
regionally proximal. For an abelian acting group, the paragraph preceding
Lemma~5 and Lemma~5(iv) in
\cite[pp.~330--331]{AuslanderGroup04} show that the algebraic hypothesis
in Auslander's finite regional proximality theorem is automatic. Therefore
\cite[Theorem~8, p.~332]{AuslanderGroup04} shows that the entire tuple is
jointly regionally proximal.
\end{proof}

The following result plays an important role in the proof of \cref{main2}.

\begin{lemma}\label{lem:full-fibre-returns}
Let $(X,T)$ be a minimal system with maximal equicontinuous
factor $\pi:(X,T)\to (Z,R)$. If $U,V\subset X$ are nonempty open sets and
\[
 \pi(V)=Z,
\]
then
\[
 N_T(U,V):=\{n\in\N:U\cap T^{-n}V\neq\emptyset\}
\]
is thickly syndetic.
\end{lemma}

\begin{proof}
Fix a nonempty finite set
$F=\{f_1,\ldots,f_s\}\subset\Nzero$ and a point $z\in Z$.
For each $i$, choose $w_i\in V$ with
\[
 \pi(w_i)=R^{f_i}z,
\]
where $R$ is the transformation induced by $T$ on $Z$, and put
$v_i=T^{-f_i}w_i$. Then $\pi(v_i)=z$ for every $i$.
By \cref{lem:Auslander-joint},
\[
 (v_1,\ldots,v_s)\in Q^{(s)}(X,T).
\]

Choose open neighborhoods
\[
 v_i\in G_i\subset T^{-f_i}V\qquad(1\leq i\leq s).
\]
By minimality and compactness, there is a finite set $E\subset\Nzero$ such
that
\[
 X=\bigcup_{e\in E}T^{-e}U.
\]
Consequently,
\[
 \mathcal D=\bigcup_{e\in E}\underbrace{T^{-e}U\times\cdots\times T^{-e}U}_{s\text{ times }}
\]
is an open neighborhood of the diagonal in $X^s$. Joint regional
proximality gives points $v_i'\in G_i$ and $k\in\Z$ such that
\[
 (T^kv_1',\ldots,T^kv_s')\in\mathcal D.
\]
Thus, for some $e\in E$, all the points $T^{k+e}v_i'$ lie in $U$.
After shrinking around the $v_i'$, we obtain nonempty open sets
$G_i'\subset G_i$ and one integer $h=k+e$ such that
\begin{equation}\label{eq:simultaneous-compression}
 T^hG_i'\subset U\qquad(1\leq i\leq s).
\end{equation}

Choose $x_1\in G_1'$. Since $(X,T)$ is minimal, for each $2\leq i\leq s$
there is $a_i\in\Nzero$ such that
\[
 x_i:=T^{a_i}x_1\in G_i';
\]
put $a_1=0$. The point
\[
 \mathbf x=(x_1,\ldots,x_s)
\]
is minimal under the diagonal action $T^{(s)}$, because its orbit closure
is the graph system
\[
 \{(x,T^{a_2}x,\ldots,T^{a_s}x):x\in X\},
\]
which is conjugate to $(X,T)$. By
\eqref{eq:simultaneous-compression}, $(T^{(s)})^h\mathbf x\in U^s$.
Thus $U^s$ meets the minimal graph system in a nonempty relatively open
set. The restriction of $T^{(s)}$ to this graph system is a minimal
homeomorphism, so its inverse is minimal as well. Therefore
\[
 A=\{n\in\N:(T^{(s)})^{-n}\mathbf x\in U^s\}
\]
is syndetic.

If $n\in A$, let $u_i=T^{-n}x_i\in U$. Since
$x_i\in G_i\subset T^{-f_i}V$,
\[
 T^{f_i+n}u_i=T^{f_i}x_i\in V.
\]
Hence $f_i+n\in N_T(U,V)$ for every $i$, and so
\[
 F+n\subset N_T(U,V).
\]
Thus the occurrence set
\[
 \{t\in\N:F+t\subset N_T(U,V)\}
\]
contains the syndetic set $A$. Since $F$ was arbitrary,
$N_T(U,V)$ is thickly syndetic.
\end{proof}

\subsection{Ultrafilters}
We shall use standard facts about ultrafilters (for more details, see \cite[Chapters~2--4 and 19]{HS}).

Let $\mathcal P(\Nzero)$ be the power set of $\Nzero$.  An ultrafilter on
$\Nzero$ is a family $p\subseteq\mathcal P(\Nzero)$ containing no empty
set, closed under finite intersections and supersets, and containing
exactly one of $A$ and $\Nzero\setminus A$ for every
$A\subseteq\Nzero$.  We identify $n\in\Nzero$ with the \emph{principal
	ultrafilter} $\{A\subseteq\Nzero:n\in A\}$; the remaining ultrafilters are
called \emph{nonprincipal ultrafilters}.  The space of all ultrafilters is denoted by $\bN$,
and
$
\Nstar=\bN\setminus\Nzero.
$
The space $\bN$ is a compact Hausdorff space, with the clopen basis
$\{p\in\bN:A\in p\}$, where $A\subseteq\Nzero$.

For $A\subseteq\Nzero$ and $n\in\Nzero$, put
$A-n=\{m\in\Nzero:m+n\in A\}$.  Addition on $\Nzero$ extends to $\bN$
by
\begin{equation}\label{eq:beta-add}
A\in p+q
\quad\Longleftrightarrow\quad
\{n\in\Nzero:A-n\in q\}\in p,
\end{equation}
see \cite[Theorem~4.12(b)]{HS}. With this operation, $(\bN,+)$ is a compact right-topological semigroup \cite[Theorems~3.28, 4.1, and 4.4]{HS}.

For a sequence $\{x_n\}_{n\in \Nzero}$ in a compact Hausdorff space $X$ and
$p\in\bN$, write $p\!\operatorname{-lim}_n x_n$ for the unique point
$x\in X$ such that
$\{n\in \Nzero:x_n\in U\}\in p$ for every nonempty open neighborhood $U$ of $x$. If $(X,T)$ is a
system, then its forward iterates define an action of $\beta\Nzero$ on
$X$ by putting
$
px=p\!\operatorname{-lim}_{n}T^n x.
$
By \cite[Theorem~19.11 and Remark~19.13]{HS}
\begin{equation}\label{eq:action}
(p+q)x=p(qx).
\end{equation}
If $\pi:(X,T)\to(Y,S)$ is a factor map, then
$\pi(T^n x)=S^n\pi(x)$, and the preservation of $p$-limits under
continuous maps gives
\begin{equation}\label{eq:factor-action}
   \pi(px)=p\pi(x);
\end{equation}
see \cite[Theorem~3.49]{HS}.

For $k\in\N$, let $D_k:\bN\to\bN$ be the continuous extension of
$n\mapsto kn$. By \cite[Lemma~3.30]{HS},
\begin{equation}\label{eq:Dk-member}
A\in D_kp
\quad\Longleftrightarrow\quad
\{n\in\Nzero:kn\in A\}\in p.
\end{equation}
Since multiplication by $k$ is a semigroup homomorphism and its range lies
in the topological center of $\bN$,
\begin{equation}\label{eq:Dk-hom}
   D_k(p+q)=D_kp+D_kq
\end{equation}
by \cite[Corollary~4.22]{HS}. We write $D=D_2$.

A nonempty subset $L$ of a semigroup $S$ is a \emph{left ideal} if
$SL\subseteq L$, and is \emph{minimal} if it contains no proper left
ideal.  An element $e\in S$ is an \emph{idempotent} if $e^2=e$.  Every
left ideal of a compact right-topological semigroup contains a minimal
left ideal, and every minimal left ideal is closed and contains an
idempotent \cite[Corollary~2.6]{HS}.

Next, we introduce some properties of ultrafilters.
\begin{lemma}\label{lem:ultrafilter-facts}
	Let $p,q\in\bN$ and $k\in\N$. Then:
	\begin{itemize}
		\item[(i)] For every $r\in\Nzero$, one has $r+p=p+r$.
		\item[(ii)] If $q\in\Nstar$, then $p+q\in\Nstar$; if
		$p\in\Nstar$, then $D_kp\in\Nstar$.
		\item[(iii)] Every $q\in\Nstar$ has a unique representation
		\begin{equation}\label{eq:parity}
		q=\varepsilon+D_kp,
		p\in\Nstar,
		\varepsilon\in\{0,1,\ldots,k-1\}.
		\end{equation}
	\end{itemize}
\end{lemma}

\begin{proof}
For (i), fix $A\subseteq\Nzero$.  By \eqref{eq:beta-add},
\[
   A\in r+p
   \Longleftrightarrow A-r\in p
   \Longleftrightarrow A\in p+r.
\]

For (ii), let $F\subseteq\Nzero$ be finite and let $q$ be a nonprincipal ultrafilter.  For every
$n\in\Nzero$, the set $F-n$ is finite, so $F-n\notin q$.  Hence
$F\notin p+q$ by \eqref{eq:beta-add}.  Thus $p+q$ contains no finite set and
is a nonprincipal ultrafilter.  If $p$ is a nonprincipal ultrafilter and
$F\in D_kp$ were finite, then
$\{n:kn\in F\}\in p$ by \eqref{eq:Dk-member}, although this set is finite,
a contradiction.  Hence $D_kp$ is a nonprincipal ultrafilter.

For (iii), the ultrafilter $q$ contains exactly one of the residue classes
$k\Nzero+\varepsilon$, $0\leq\varepsilon<k$. Define an
ultrafilter $p$ on $\Nzero$ by
\[
   B\in p
   \quad\Longleftrightarrow\quad
   \{kn+\varepsilon:n\in B\}\in q
   \qquad(B\subseteq\Nzero).
\]
This is an ultrafilter because $n\mapsto kn+\varepsilon$ is a bijection from
$\Nzero$ onto the $q$-large set $k\Nzero+\varepsilon$. Since that residue
class belongs to $q$, for every $A\subseteq\Nzero$,
\[
   A\in q
   \quad\Longleftrightarrow\quad
   \{n:kn+\varepsilon\in A\}\in p.
\]
Together with \eqref{eq:Dk-member}, this gives
$q=\varepsilon+D_kp$. If $p$ were principal, then so would be $q$, so
$p\in\Nstar$.

For uniqueness, $\varepsilon$ is determined uniquely by which residue class
belongs to $q$.  Once $\varepsilon$ is fixed, the preceding formula gives,
for every $B\subseteq\Nzero$,
\[
   B\in p
   \quad\Longleftrightarrow\quad
   \{kn+\varepsilon:n\in B\}\in q,
\]
so it also determines $p$ uniquely.
\end{proof}

The following standard result in topological dynamics follows, for example,
from \cite[pp.~38--40]{Blass}. We include a proof for completeness.

\begin{lemma}\label{lem:omega-ultrafilter}
	For every topological dynamical system $(X,T)$ and every $x\in X$,
	$\omega(x)$ is a subsystem and
	\begin{equation}\label{eq:omega-ultrafilter}
	\omega(x)=\{px:p\in\Nstar\}.
	\end{equation}
\end{lemma}

\begin{proof}
	If $y=\lim_iT^{n_i}x\in\omega(x)$ with $n_i\to\infty$, then
	$T^{-1}y=\lim_iT^{n_i-1}x\in\omega(x)$. Together with forward
	invariance, this shows that $\omega(x)$ is a subsystem. Let
	$p\in\Nstar$. Then for any $N$, $\{n:n\geq N\}\in p$. Therefore, by \eqref{eq:omega-def},
	$px\in\omega(x)$.
	
	Conversely, let $y\in\omega(x)$.  For each nonempty open neighborhood $U$ of $y$, put
	\[
	N_U=\{n\in\Nzero:T^nx\in U\}.
	\]
	Let $U_1,\ldots,U_k$ be nonempty open neighborhoods of $y$ and let
	$F\subset\Nzero$ be nonempty and finite.  Set
	$U=\bigcap_{i=1}^k U_i$ and $N=1+\max(F\cup\{0\})$.  Since
	$y\in\overline{\{T^nx:n\geq N\}}$, there is $n\geq N$ such that
	$T^nx\in U$.  Hence,
	\[
	n\in\bigcap_{i=1}^kN_{U_i}\cap(\Nzero\setminus F).
	\]
	Thus the family
	$
	\{N_U:U\text{ is a neighborhood of }y\}
	\cup
	\{\Nzero\setminus F:F\subset\Nzero\text{ finite}\}
	$
	has the finite-intersection property.  Extend it to an ultrafilter $p$.
	Then $p\in\Nstar$.
	Since $N_U\in p$ for every nonempty open neighborhood $U$ of $y$, $px=y$.
	This completes the proof.
\end{proof}

\subsection{Zero-dimensional separation and coding}

The following two elementary separation and coding lemmas will be used in
both proofs of the main results.

\begin{lemma}\label{lem:clopen-domain}
	Let $K$ be a compact zero-dimensional Hausdorff space, and let
	$F:K\to K$ be a continuous involution without fixed points.  If a compact set
	$M\subseteq K$ satisfies $M\cap FM=\emptyset$, then there is a clopen set
	$H\subseteq K$ such that $M\subseteq H$ and
	\begin{equation}\label{eq:clopen-domain}
	K=H\sqcup FH.
	\end{equation}
In particular, taking $M=\emptyset$ gives a clopen fundamental domain for
$J$.
\end{lemma}
\begin{proof}
	For each $x\in M$, choose a clopen neighborhood $W_x$ disjoint from $FM$.
	A finite union of these neighborhoods gives a clopen set $W$ with
	$M\subseteq W$ and $W\cap FM=\emptyset$.  Put
	$
	U=W\cap F(K\setminus W).
	$
	If $x\in M$, then $Fx\in FM\subseteq K\setminus W$. So,
	$x\in F(K\setminus W)$. Hence, $M\subseteq U$.  Also $U\subseteq W$ and
	$FU\subseteq K\setminus W$. So, $U\cap FU=\emptyset$.
	
	Let
	$
	K_0=K\setminus(U\cup FU).
	$
	Then $K_0$ is compact, clopen, and $F$-invariant.  Fix $x\in K_0$.  Choose
	disjoint open sets $O,O'$ with $x\in O$ and $Fx\in O'$, and put
	$V=O\cap F(O')$.  Then $x\in V$ and $V\cap FV=\emptyset$. Choose a relatively clopen set $V_x\subseteq K_0$ with
	$x\in V_x\subseteq V$. Then $V_x\cap FV_x=\emptyset$.
	
	By compactness, finitely many sets $W_i=V_i\cup FV_i$ cover $K_0$.  Each
	$W_i$ is clopen and $F$-invariant. Let
	\[
	P_1=W_1\ \text{and for each}\ i>1,\ \text{put}\
	P_i=W_i\setminus\bigcup_{j<i}W_j
	\]
	and discard the empty terms.  The sets $P_i$ form a finite pairwise disjoint
	clopen $F$-invariant partition of $K_0$.  Put $H_i=P_i\cap V_i$.  Since
	$P_i\subseteq V_i\cup FV_i$, $V_i\cap FV_i=\emptyset$, and $FP_i=P_i$,
	$
	P_i=H_i\sqcup FH_i.
	$
	Finally, set
	$
	H=U\cup\bigcup_iH_i.
	$
	Then $H$ is clopen, contains $M$, and $K=H\sqcup FH$. This completes the proof.
\end{proof}

Let $(X,T)$ be a topological dynamical system, let $H\subset X$ be clopen. 
Define
\begin{equation*}
\pi_H:X\to\two^{\Nzero},\quad
\pi_H(x)(n)=1_H(T^nx).
\end{equation*}
The map $\pi_F$ is continuous and equivariant. 

\begin{lemma}\label{lem:asymptotic-binary-coding}
Let $(X,T)$ be a topological dynamical system, let $x\in X$, and put
$\Omega_x=\omega(x)$. Let $J:\Omega_x\to\Omega_x$ be a continuous
involution commuting with $T$. Suppose $H\subset\Omega_x$ is clopen and
\[
 \Omega_x=H\sqcup JH.
\]
Then there is a binary sequence $h:\Nzero\to\{0,1\}$ such that, for every
$p\in\Nstar$,
\begin{equation}\label{eq:asymptotic-code}
 ph=\pi_H(px),
 \qquad
 \pi_H(y)(n)=1_H(T^ny).
\end{equation}
Here $ph$ denotes the forward ultrafilter translate of the sequence $h$:
\[
 (ph)(n)=p\!\operatorname{-lim}_{k}h(k+n)
 \qquad(n\in\Nzero).
\]
Equivalently, $(ph)(n)=\gamma$ if and only if
$\{k\in\Nzero:h(k+n)=\gamma\}\in p$.
Thus both $ph$ and $\pi_H(px)$ in \eqref{eq:asymptotic-code} are elements
of $\{0,1\}^{\Nzero}$. On this sequence space, $J$ denotes coordinatewise
complementation.
Consequently, if $u,v\in\Nstar$ and $ux=J(vx)$, then
$uh=J(vh)$. If a minimal subsystem $L\subset\Omega_x$ is contained in
$H$, then $h^{-1}(1)\cap\N$ is thick in $\N$, with blocks occurring
arbitrarily far to the right.
\end{lemma}

\begin{proof}
The function $1_H$ is continuous on the closed subset $\Omega_x$. Extend it
to a continuous $F:X\to[0,1]$ and put
$\mathcal O=\{z:F(z)>1/2\}$. Since $F$ takes only the values $0$ and $1$ on
$\Omega_x$,
\[
 \mathcal O\cap\Omega_x=H,
 \qquad
 \partial\mathcal O\cap\Omega_x=\emptyset.
\]
Define $h(n)=1_{\mathcal O}(T^nx)$. For $p\in\Nstar$, one has
$px\in\Omega_x$. The indicator $1_{\mathcal O}$ is continuous at every point
of $\Omega_x$, so for each $n\in\Nzero$,
\[
 (ph)(n)
 =p\text{-}\lim_k1_{\mathcal O}(T^{k+n}x)
 =1_H(T^n(px)).
\]
This proves \eqref{eq:asymptotic-code}. If $ux=J(vx)$, then
\[
 uh=\pi_H(ux)=\pi_H(J(vx))=J\pi_H(vx)=J(vh).
\]
Finally, let $L\subset H$ be minimal. Given $R,N\in\N$ and $y\in L$, the
open set
\[
 \bigcap_{j=0}^{R}T^{-j}\mathcal O
\]
contains $y$, because $T^jL=L\subset H$. Since $y\in\omega(x)$, some
$n\ge N$ satisfies $T^nx$ in this intersection. Hence
$h(n)=\cdots=h(n+R)=1$.
\end{proof}

\begin{lemma}
\label{lem:equivariant-circle-extension}
Let $K$ be a compact zero-dimensional metrizable space, let $J:K\to K$ be a continuous involution without fixed points, and let
$E\subset K$ be closed and $J$-invariant. If $\psi:E\to\Torus$ is continuous
and
\[
 \psi(Jx)=\psi(x)+\tfrac12\qquad(x\in E),
\]
then there is a continuous extension $\bar\psi:K\to\Torus$ satisfying
\[
 \bar\psi(Jx)=\bar\psi(x)+\tfrac12\qquad(x\in K).
\]
\end{lemma}

\begin{proof}
Apply \cref{lem:clopen-domain} with $M=\emptyset$ and choose
a clopen fundamental domain $D$. Put $E_D=E\cap D$. Identify $\Torus$ with
the unit circle in $\mathbb C$, while retaining additive notation for its
group law. By the Tietze extension theorem,
the two real coordinates of $\psi|_{E_D}$ extend to a continuous map
$f:D\to\mathbb C$. Since $|f|=1$ on $E_D$, the open set
$\{x\in D:|f(x)|>1/2\}$ contains $E_D$. Zero-dimensionality and compactness
give a clopen set $D_0$ with
\[
 E_D\subset D_0\subset\{x\in D:|f(x)|>1/2\}.
\]
Define
\[
 \varphi(x)=\frac{f(x)}{|f(x)|}\quad(x\in D_0),
 \qquad
 \varphi(x)=1\quad(x\in D\setminus D_0).
\]
Then $\varphi:D\to\Torus$ is continuous and extends $\psi|_{E_D}$. Define
\[
 \bar\psi(x)=\varphi(x)\quad(x\in D),
 \qquad
 \bar\psi(x)=\varphi(Jx)+\tfrac12\quad(x\in JD).
\]
The two definitions are made on disjoint clopen sets. If $x\in E\cap JD$,
write $x=Jy$ with $y\in E_D$; then
$\bar\psi(x)=\psi(y)+1/2=\psi(Jy)=\psi(x)$. Thus $\bar\psi$ is the required
extension.
\end{proof}

\subsection{Admissible finite partitions}

Hindman~\cite[Definition~2.2]{Neil79}
introduced admissible partitions, requiring one cell to contain arbitrarily
long arithmetic progressions of even integers with a fixed common
difference. That is,

\begin{defn}\cite[Definition~2.2]{Neil79}\label{defn-hindman}
A partition $\N=C_1\sqcup \cdots \sqcup C_r$  is {\em admissible} if there exist $i\in \{1,\ldots,r\}$ and $d\in \N$ such that, for each $n\in \N$, there  is an even integer $x\in \N$ with $\{x+kd:0\le k\le n\}\subseteq C_i$.
\end{defn}

A special case of an admissible partition is one which has some cell
including arbitrarily long blocks, i.e. it is thick.

Hindman proved that every admissible two-cell partition contains
$B+B$ for some infinite $B$. That is,

\begin{thm}\cite[Corollary~2.10]{Neil79}\label{thm:Hindman-even-ap}
	Let $\N=C_1\sqcup C_2$ be an admissible partition.  Then there exist $i\in \{1,2\}$ and an infinite $B\subseteq\N$ such that
	$B+B\subseteq C_i$.
\end{thm}

We generalize Definition \cref{defn-hindman} as follows.

\begin{defn}\label{defn4-1}
	Fix $(m,\ell)\in\N^2$. Let
	$\N=C_1\sqcup\cdots\sqcup C_r$ be an $r$-coloring. We say that it is
	{\em $(m,\ell)$-admissible} if there are $i\in\{1,\ldots,r\}$ and
	$d\in(m+\ell)\N$ such that, for every $n\in\N$, there is $x\in\N$
	satisfying
	\[
	(m+\ell)x,(m+\ell)x+d,\ldots,(m+\ell)x+nd\in C_i.
	\]
\end{defn}

We will need the following generalization of \cref{thm:Hindman-even-ap} in the proof of \cref{main2}. 

\begin{thm}\label{thm2-1}
	Fix $(m,\ell)\in\N^2$. Let
	$\N=C_1\sqcup C_2$ be an $(m,\ell)$-admissible $2$-coloring. Then
	there is an infinite $B\subseteq\N$ such that
	\[
	(m+\ell)B\cup\{mx+\ell y:x,y\in B,\ x<y\}
	\]
	is monochromatic.
\end{thm}

The theorem is proved in \cref{appA}. When $m=\ell=1$, its role is played
by \cref{thm:Hindman-even-ap}.

\section{Proof of \texorpdfstring{\cref{main1}}{the unweighted theorem}}\label{sec3}
In this section, we prove \cref{main1} by a contradiction argument. Throughout the section, let
$b:\N\to\two$ be a coloring satisfying the standing hypothesis:
\begin{equation*}\tag{H}\label{hyp:counterexample}
\text{no infinite $B\subseteq\N$ has monochromatic $B+B$ under $b$.}
\end{equation*}
Choose $b(0)$ arbitrarily.  The extension to $\Nzero$ still satisfies
\eqref{hyp:counterexample}: if an infinite witness contained $0$, deleting
$0$ would leave an infinite witness in $\N$.
Define
\begin{equation}\label{eq:def-c}
c(n)=b(2n)\ \text{for each}\ n\in\Nzero.
\end{equation}
Assign arbitrary values to $c(n)$ for $n<0$.

\noindent\textbf{A brief outline of the proof.}\enspace
Suppose, toward a contradiction, that $b:\N\to\two$ has no
monochromatic $B+B$, and set $c(n)=b(2n)$.  We encode all affine
samplings $n\mapsto c(mn+r)$ as a point $\mathbf c$ in a product of
two-sided shifts. The semigroup
$\bN$ acts on the orbit closure of $\mathbf c$, and
$D:\bN\to\bN$ denotes the continuous extension of $n\mapsto2n$.
If $J$ is coordinatewise color complementation, then
\cref{prop:key-identity} gives
\[
   (p+p)\mathbf c
   =
   J\bigl((Dp)\mathbf c\bigr)
   \qquad (p\in\Nstar).
\]
In the proof, membership in $Dp$ controls the doubles $2v_i$, while
membership in $p+p$ controls the sums $v_i+v_j$ for $i<j$.

By \cref{prop:disjoint-minimal}, every minimal subsystem
$M\subseteq\omega(\mathbf c)$ satisfies $M\cap JM=\emptyset$.
A clopen separation of $M$ and $JM$ then produces a factor coloring $h$
with no monochromatic $B+B$ and with a thick color class; see
\cref{prop:counterexample-factor}.  This class contains arbitrarily long
progressions of even integers with common difference $2$, so
\cref{thm:Hindman-even-ap} yields a monochromatic $B+B$, a contradiction.

\subsection{The affine encoding}
Let
$
\mathcal I=\{(s,r):s\in\N,\ 0\leq r<s\},
\mathcal X=(\two^{\Z})^{\mathcal I}.
$ A point in $\mathcal X$ is denoted by
$(x_{s,r})_{(s,r)\in\mathcal I}$.
Give $\mathcal X$ the product topology and define
$T:\mathcal X\to\mathcal X$ by
$$
(Tx)_{s,r}(t)=x_{s,r}(t+1),
\qquad (s,r)\in\mathcal I,\ t\in\Z.
$$
The set $\mathcal I\times\Z$ is countable, and $\mathcal X$ is naturally
homeomorphic to $\two^{\mathcal I\times\Z}$. Thus $\mathcal X$ is
compact by Tychonoff's theorem, metrizable because the index set is
countable, and zero-dimensional because its cylinder sets are clopen.
Moreover, $T$ is a homeomorphism.

Define $\mathbf c\in\mathcal X$ by
\[
\mathbf c_{s,r}(t)=c(st+r)
\qquad((s,r)\in\mathcal I,\ t\in\Z),
\]
and put
\begin{equation}\label{eq:bold-c}
X_{\mathbf c}=\overline{\{T^n\mathbf c:n\in\Z\}}.
\end{equation}
Thus $(X_{\mathbf c},T)$ is a topological dynamical system, and
$X_{\mathbf c}$ is zero-dimensional.

Define
$J:\mathcal X\to\mathcal X$ by
\begin{equation}\label{eq:J}
(Jx)_{s,r}(t)=1-x_{s,r}(t).
\end{equation}
Then $J$ is a fixed-point-free continuous involution and  $JT=TJ$.

Next, we introduce a map $\Delta_2$ on $\mathcal X$ and establish some of
its properties. For $k\in\Z$ and $j\in\{0,1\}$, define
$\Delta_2:\mathcal X\to\mathcal X$ by
	\begin{equation}\label{eq:Delta-def}
		(\Delta_2 x)_{s,r}(2k+j)=x_{2s,sj+r}(k),
		\qquad (s,r)\in\mathcal I.
	\end{equation}
\begin{lemma}\label{lem:Delta}
The map $\Delta_2$ is continuous and
	\begin{equation}\label{eq:Delta-relations}
		\Delta_2\mathbf c=\mathbf c,\
		\Delta_2 T=T^2\Delta_2,\ \text{and}\
		\Delta_2 J=J\Delta_2.
	\end{equation}
	Moreover, for every $p\in\bN$ and $x\in\mathcal X$,
	\begin{equation}\label{eq:Delta-p}
		\Delta_2(px)=(Dp)(\Delta_2 x).
	\end{equation}
In particular, $\Delta_2(X_{\mathbf c})\subseteq X_{\mathbf c}$.
\end{lemma}

\begin{proof}
Since $0\leq mj+r<2m$, the coordinate $(2m,mj+r)$ in
\eqref{eq:Delta-def} belongs to $\mathcal I$. Every output coordinate is a single
coordinate projection of $x$, so $\Delta_2$ is continuous.

For $(m,r)\in\mathcal I$, $k\in\Z$, and $j\in\{0,1\}$,
\begin{align*}
   (\Delta_2\mathbf c)_{m,r}(2k+j)
   &=\mathbf c_{2m,mj+r}(k)
   =c(2mk+mj+r)\\
   &=c\bigl(m(2k+j)+r\bigr)\\
   &=\mathbf c_{m,r}(2k+j),
\end{align*}
so $\Delta_2\mathbf c=\mathbf c$.  Similarly,
\begin{align*}
   (\Delta_2 Tx)_{m,r}(2k+j)
   &=(Tx)_{2m,mj+r}(k)
   =x_{2m,mj+r}(k+1)\\
   &=(T^2\Delta_2 x)_{m,r}(2k+j),
\end{align*}
and
\[
   (\Delta_2 Jx)_{m,r}(2k+j)
   =1-x_{2m,mj+r}(k)
   =(J\Delta_2 x)_{m,r}(2k+j).
\]
This proves \eqref{eq:Delta-relations}.

By continuity and $\Delta_2 T^n=T^{2n}\Delta_2$,
\[
   \Delta_2(px)
   =p\!\operatorname{-lim}_n\Delta_2(T^nx)
   =p\!\operatorname{-lim}_nT^{2n}(\Delta_2 x)
   =(Dp)(\Delta_2 x),
\]
which proves \eqref{eq:Delta-p}.  Finally,
$\Delta_2(T^n\mathbf c)=T^{2n}\mathbf c$ for every $n\in\Z$, and
continuity gives
\[
   \Delta_2(X_{\mathbf c})
   \subseteq\overline{\{T^{2n}\mathbf c:n\in\Z\}}
   \subseteq X_{\mathbf c}.
\] This completes the proof.
\end{proof}

\subsection{A key identity for the encoded point}

The following proposition characterizes the point $(p+p){\mathbf c}$.
\begin{prop}\label{prop:key-identity}
	Under the standing hypothesis \eqref{hyp:counterexample}, for every
	$p\in\Nstar$,
	\begin{equation}\label{eq:key-identity}
	(p+p)\mathbf c=J(Dp)\mathbf c.
	\end{equation}
\end{prop}
Before proving it, we need two lemmas.
\begin{lemma}\label{lem:no-shifted-sumset}
	Under the standing hypothesis \eqref{hyp:counterexample}, there do not exist $R\in\Z$, a strictly increasing sequence
	$0\leq n_1<n_2<\cdots$, and $\gamma\in\two$ such that for all $i\leq j$,
	\begin{equation}\label{eq:shifted-sumset}
	c(n_i+n_j+R)=\gamma
	\end{equation}
\end{lemma}
\begin{proof}
If \eqref{eq:shifted-sumset} holds, put $m_i=2n_i+R$. Then
$\{m_i\}_{i\ge 1}$ is strictly increasing. Discarding finitely many terms,
we may assume that all $m_i$ lie in $\N$. For all remaining $i\leq j$,
	$
	b(m_i+m_j)
	=b\bigl(2(n_i+n_j+R)\bigr)
	=c(n_i+n_j+R)
	=\gamma.
	$
This contradicts \eqref{hyp:counterexample}.
\end{proof}
\begin{lemma}\label{lem:diagonal-recursion}
	Let $A\subseteq\Nzero$ and $p\in\Nstar$.  If
	$A\in p+p\ \text{and}\
	A\in Dp, $
	then there is a strictly increasing sequence of positive integers
	$v_1<v_2<\cdots$ such that for all $i\le j$,
	\begin{equation}\label{eq:diagonal-recursion}
	v_i+v_j\in A.
	\end{equation}
\end{lemma}
\begin{proof}
	By \eqref{eq:beta-add} and \eqref{eq:Dk-member},
	$
	K=\{n:A-n\in p\}\in p\ \text{and}\
	H=\{n:2n\in A\}\in p.
	$
	Choose $v_1\in H\cap K\cap\N$.  Suppose that
	$v_1<\cdots<v_{s-1}$ have been chosen and that
	\eqref{eq:diagonal-recursion} holds for $1\leq i\leq j<s$.  Since
	$v_i\in K$, one has $A-v_i\in p$.  Hence,
	\[
	E_s=H\cap K\cap\bigcap_{i=1}^{s-1}(A-v_i)
	\cap\{n\in \N:n>v_{s-1}\}\in p.
	\]
	Choose
	$v_s\in E_s$.  Then $v_s>v_{s-1}$, $v_s\in H$ gives
	$2v_s\in A$, and $v_s\in A-v_i$ gives $v_i+v_s\in A$ for $i<s$.
	The induction is complete.
\end{proof}
Now, we prove \cref{prop:key-identity}.
\begin{proof}[Proof of \cref{prop:key-identity}]
		Fix $p\in\Nstar$ and suppose that \eqref{eq:key-identity} fails.  Then
		there are $(s,r)\in\mathcal I$ and $t\in\Z$ such that
		$
		((p+p)\mathbf c)_{s,r}(t)
		\neq 1-((Dp)\mathbf c)_{s,r}(t).
		$
Since both values belong to $\two$, they must be equal; write
				$$
		((p+p)\mathbf c)_{s,r}(t)
		=((Dp)\mathbf c)_{s,r}(t)
		:=\gamma.
		$$
		Let
		$
		U=\{x\in\mathcal X:x_{s,r}(t)=\gamma\}
		$
		and put
		\[
		A=\{n\in\Nzero:T^n\mathbf c\in U\}
		=\{n\in\Nzero:c(s(n+t)+r)=\gamma\}.
		\]
		Then $A\in p+p$ and $A\in Dp$.
		By \cref{lem:diagonal-recursion}, choose a strictly increasing sequence of
		positive integers $v_1<v_2<\cdots$ such that $v_i+v_j\in A$ for all $i\leq j$.  Set
		$
		N_i=sv_i,
		R=st+r.
		$
		Then
		$
		c(N_i+N_j+R)
		=c\bigl(s(v_i+v_j+t)+r\bigr)
		=\gamma
		$ for all $i\le j$.
	This contradicts \cref{lem:no-shifted-sumset}. This completes the proof.
\end{proof}

\subsection{Minimal subsystems of \texorpdfstring{$\omega(\mathbf c)$}{the omega-limit set}}

Let
$\Omega=\omega(\mathbf c)\subseteq X_{\mathbf c}$.
By \cref{lem:omega-ultrafilter},
\begin{equation}\label{eq:Omega-ultrafilter}
\Omega=\{p\mathbf c:p\in\Nstar\}.
\end{equation}
Choose a minimal subsystem $M\subseteq\Omega$.

First, we prove that $\Omega$ is $J$-invariant.
\begin{lemma}\label{lem:J-Omega}
	Under the standing hypothesis \eqref{hyp:counterexample}, one has
	$J\Omega=\Omega$.
\end{lemma}
\begin{proof}
	Let $x\in\Omega$.  By \eqref{eq:Omega-ultrafilter}, write
	$x=q\mathbf c$ with $q\in\Nstar$.  By
	\cref{lem:ultrafilter-facts}(iii), there are $p\in\Nstar$ and
	$\varepsilon\in\{0,1\}$ such that $q=\varepsilon+Dp$.  Therefore,
	$$
	Jx
	=J\bigl((\varepsilon+Dp)\mathbf c\bigr)\\
	=T^\varepsilon J(Dp)\mathbf c\\
	=T^\varepsilon(p+p)\mathbf c\\
	=(\varepsilon+p+p)\mathbf c,
	$$
	where the third equality uses \cref{prop:key-identity}.  By
	\cref{lem:ultrafilter-facts}(ii), $p+p\in\Nstar$ and then
	$\varepsilon+p+p\in\Nstar$.  Thus $Jx\in\Omega$. So,
	$J\Omega\subseteq\Omega$.  Since $J^2=\operatorname{id}$, the equality follows. This completes the proof.
\end{proof}
Because $J$ is a homeomorphism and $JT=TJ$, the set $JM$ is also a minimal
subsystem. Indeed, if $Y\subseteq JM$ is a nonempty subsystem, then
$JY\subseteq M$ is also a nonempty subsystem;
minimality of $M$ gives $JY=M$, and hence $Y=JM$.

\begin{prop}\label{prop:disjoint-minimal}
	Under the standing hypothesis \eqref{hyp:counterexample}, one has
	$M\cap JM=\emptyset$.
\end{prop}

\begin{proof}
  If $M\cap JM\neq\emptyset$, the two minimal subsystems are equal.  Suppose
therefore that $M=JM$.
Set
\[
   \mathcal I_M=\{p\in\bN:p\mathbf c\in M\}.
\]
This set is nonempty by \eqref{eq:Omega-ultrafilter}, and it is closed by
continuity of $p\mapsto p\mathbf c$.  If $p\in\mathcal I_M$ and
$s\in\bN$, then
\[
   (s+p)\mathbf c=s(p\mathbf c)\in M.
\]
Indeed, all points $T^n(p\mathbf c)$ lie in the closed set $M$, so their
$s$-limit also lies in $M$.  Hence $\mathcal I_M$ is a left ideal of $\bN$.

Choose a minimal left ideal $L\subseteq\mathcal I_M$ and an idempotent
$e\in L$.  We claim that $e\in\Nstar$.  If $e$ were principal, then
$e+e=e$ would force $e=0$.  Since $L$ is a left ideal and $0\in L$,
\[
   \bN=\bN+0\subseteq L,
\]
so $L=\bN$.  On the other hand, \cref{lem:ultrafilter-facts}(ii) shows that
$\Nstar$ is a nonempty proper left ideal of $\bN$, contradicting the
minimality of $L$.  Thus $e$ is a nonprincipal ultrafilter.

Put $y=e\mathbf c\in M$.  From $e+e=e$ and
\cref{prop:key-identity},
\[
   y=(e+e)\mathbf c=J(De)\mathbf c,
\]
so
\begin{equation}\label{eq:De-c}
   (De)\mathbf c=Jy.
\end{equation}
By \eqref{eq:Delta-p} and $\Delta_2\mathbf c=\mathbf c$,
\begin{equation}\label{eq:Delta-y}
   \Delta_2 y=(De)\mathbf c=Jy,
   \qquad
   \Delta_2(Jy)=J\Delta_2 y=y.
\end{equation}

Since $M=JM$, one has $Jy\in M$.  The continuous map
\[
   \theta_y:\bN\longrightarrow M,
   \qquad
   \theta_y(p)=py,
\]
has image equal to $M$: its compact image is closed, contains the forward
orbit of $y$, and is contained in its closure, which is $M$ by minimality.
Thus there is $s\in\bN$ such that
\begin{equation}\label{eq:s-y}
   sy=Jy.
\end{equation}
Let $q=s+e$.  Since the right factor $e$ is a nonprincipal ultrafilter,
\cref{lem:ultrafilter-facts}(ii) gives $q\in\Nstar$, and
\begin{equation}\label{eq:q-c}
   q\mathbf c=s(e\mathbf c)=sy=Jy.
\end{equation}
Moreover, $ey=(e+e)\mathbf c=y$, and \eqref{eq:factor-action} applied to
$J$ gives
\[
   e(Jy)=J(ey)=Jy,
   \qquad
   s(Jy)=J(sy)=y.
\]
Consequently,
\begin{equation}\label{eq:qplusq-c}
   (q+q)\mathbf c
   =q(q\mathbf c)
   =q(Jy)
   =s(e(Jy))
   =y.
\end{equation}
On the other hand, \eqref{eq:Delta-p}, \eqref{eq:q-c}, and
\eqref{eq:Delta-y} give
\begin{equation}\label{eq:Dq-c}
   (Dq)\mathbf c
   =\Delta_2(q\mathbf c)
   =\Delta_2(Jy)
   =y.
\end{equation}
Applying \cref{prop:key-identity} to the same nonprincipal ultrafilter $q$ yields
\[
   y=(q+q)\mathbf c=J(Dq)\mathbf c=Jy,
\]
contrary to the fact that $J$ has no fixed point. This completes the proof.
\end{proof}

\subsection{Constructing the new coloring \texorpdfstring{$h$}{h} and its properties}

By \cref{lem:J-Omega,prop:disjoint-minimal,lem:clopen-domain}, there is
a clopen set $H\subseteq\Omega$ such that $M\subseteq H$ and
\begin{equation}\label{eq:H-domain}
\Omega=H\sqcup JH.
\end{equation}
Apply \cref{lem:asymptotic-binary-coding} to $(X_{\mathbf c},T)$,
$\mathbf c$, $\Omega$, $J$, and $H$. We obtain a binary coloring
$h:\Nzero\to\two$ for which $h^{-1}(1)\cap\N$ is thick in $\N$. Moreover,
\cref{prop:key-identity} and the coding conclusion give, for every
$p\in\Nstar$,
\begin{equation}\label{eq:h-identity}
(p+p)h=J(Dp)h.
\end{equation}
Indeed, $p+p$ and $Dp$ are free by
\cref{lem:ultrafilter-facts}(ii), so both orbit limits belong to $\Omega$.

\begin{prop}\label{prop:counterexample-factor}
Under the standing hypothesis \eqref{hyp:counterexample}, the coloring
$h:\Nzero\to\two$ obtained above has the following properties:
\begin{enumerate}
\item[(i)] no infinite set $B\subseteq\N$ has $B+B$ monochromatic under $h$;
\item[(ii)] the set $h^{-1}(1)\cap\N$ is thick in $\N$. More precisely,
      for every $L,N\in\N$, there is $n\geq N$ such that
      $h(n)=h(n+1)=\cdots=h(n+L)=1$.
\end{enumerate}
\end{prop}

\begin{proof}
	For (i), suppose that there are an infinite $B\subseteq\N$ and
	$\gamma\in\two$ such that
	\begin{equation}\label{eq:h-mono}
		h(B+B)=\gamma.
	\end{equation}
	Extend the family
	$
	\{B\}\cup
	\{\Nzero\setminus F:F\subseteq\Nzero\text{ finite}\}
	$
	to an ultrafilter $p$. Then $p\in \Nstar$.  Let
	$A_\gamma=\{n\in\Nzero:h(n)=\gamma\}$.  From \eqref{eq:h-mono},
	$
	B\subseteq\{n\in\Nzero:2n\in A_\gamma\}.
	$
	Since $B\in p$, equation \eqref{eq:Dk-member} gives
	$A_\gamma\in Dp$.  Moreover, for every $u\in B$,
	$B\subseteq A_\gamma-u$. So, $A_\gamma-u\in p$.  Hence,
	$
	B\subseteq\{u\in \Nzero:A_\gamma-u\in p\}.
	$
	Then $\{u\in \Nzero:A_\gamma-u\in p\}\in p$. Then \eqref{eq:beta-add} gives $A_\gamma\in p+p$.
	Taking $0$-coordinate in \eqref{eq:h-identity} yields
	\[
	\gamma=((p+p)h)(0)
	=1-((Dp)h)(0)
	=1-\gamma.
	\]
	This is a contradiction.
	
	Property (ii) is part of the conclusion of
	\cref{lem:asymptotic-binary-coding}, since $M\subset H$.
\end{proof}

\subsection{Completion of the proof}
\begin{proof}[Proof of \cref{main1}]
	Suppose, toward a contradiction, that \cref{main1} is false.  Then there
	is a coloring $b:\N\to\two$ satisfying \eqref{hyp:counterexample}, so the
	preceding construction and \cref{prop:counterexample-factor} apply.
	Restrict $h$ to $\N$ and put $A_i=\{n\in\N:h(n)=i\}$ for
	$i\in\two$.  Fix $L\geq1$.  By (ii) of \cref{prop:counterexample-factor}, there is $n\geq1$
	such that
	$
	\{n,n+1,\ldots,n+2L\}\subseteq A_1.
	$
	Let $r_L$ be the least even integer not smaller than $n$.  Then
	$
	r_L,r_L+2,\ldots,r_L+2(L-1)\in A_1.
	$
Thus the hypothesis of \cref{thm:Hindman-even-ap} holds.  That theorem gives an infinite $B\subseteq\N$ for which
$B+B$ is monochromatic under $h$, contradicting (i) of \cref{prop:counterexample-factor}.
This completes the proof.
\end{proof}

\section{Proof of the weighted theorem}\label{sec4}

The proof follows the general scheme used in \cref{sec3}, but the
asymmetric coefficients require several dilation maps and an additional
analysis of the maximal equicontinuous factor.  All general
topological-dynamical, return-time, and ultrafilter facts have already
been collected in \cref{sec2}.

Fix relatively prime $m,\ell\in\N$ and suppose, toward a contradiction, that
$b:\N\to\{0,1\}$ admits no configuration of the form
\eqref{eq:weighted-target}. Choose $b(0)$ arbitrarily and regard $b$ as a
coloring of $\Nzero$. Define
\[
 c(n)=b((m+\ell)n)\qquad(n\in\Nzero),
\]
and assign arbitrary values to $c(n)$ for $n<0$. If a strictly increasing
sequence $(n_i)$ in $\Nzero$ satisfied
\begin{equation}\label{eq:c-target}
 c((m+\ell)n_i)=c(mn_i+\ell n_j)=\gamma\qquad(i<j),
\end{equation}
then, after omitting $0$ if necessary, the sequence
$a_i=(m+\ell)n_i$ would satisfy \eqref{eq:weighted-target} for $b$.
Consequently, no sequence satisfying \eqref{eq:c-target} exists.

\noindent\textbf{A brief outline of the proof.}\enspace
Starting from the counterexample coloring fixed above, we encode all
two-sided affine samplings of the rescaled coloring
in a point $\mathbf c$.  The dilation maps $\Delta_k$ on the resulting
shift space convert the absence of the desired configuration into the
ultrafilter identity
\[
 (D_mp+D_\ell p)\mathbf c
 =J(D_{m+\ell}p)\mathbf c.
\]
This identity first shows that every minimal subsystem
$M\subset\omega(\mathbf c)$ is invariant under color reversal $J$.

The two dilation systems generated by $\Delta_\ell(M)$ and
$\Delta_{m+\ell}(M)$ coincide; denote the common minimal system by $N$.
On the maximal equicontinuous factors of $M$ and $N$, the maps
$\Delta_\ell$, $\Delta_{m+\ell}$, and $J$ become affine maps.  Comparing
their translation parts shows that $J$ acts trivially on the maximal
equicontinuous factor of $N$.  The case where $m$ is odd follows directly
from the resulting torsion relations.  When $m$ is even, a circle
extension records the possible nontrivial two-torsion phase; a clopen
coding of that phase would produce a counterexample coloring with a thick
color class, contradicting the thick-cell lemma.

Finally, we choose a suitable $T^\ell$-minimal component $M_a$ and one
color $\gamma$.  At the $j$th stage, a piecewise syndetic set $P_j$ of
returns to $M_a$ enforces the diagonal condition, while a thickly syndetic
set $H_j$ preserves all previous cross conditions.  Its largeness is
proved by passing to a suitable $T^{\ell m}$-minimal component and applying
\cref{lem:power-mef,lem:full-fibre-returns}. Choosing
$r_j\in P_j\cap H_j$ recursively and putting $n_j=a+\ell r_j$ gives
\[
 c((m+\ell)n_j)=c(mn_i+\ell n_j)=\gamma\qquad(i<j),
\]
which is the required contradiction.  The general case follows by
dividing $m$ and $\ell$ by their greatest common divisor.

\noindent\textbf{Compare with case $m=\ell=1$.}\enspace
Although \cref{sec3} treats the special case $m=\ell=1$ of the present
theorem, its proof uses a shortcut that is not available for general
weights.  In that case the single dilation identity
\[
 (p+p)\mathbf c=J(Dp)\mathbf c
\]
directly forces every minimal subsystem $M\subseteq\omega(\mathbf c)$ to
be disjoint from $JM$.  A clopen separation of these two systems then
produces a counterexample coloring with a thick color class, and
\cref{thm:Hindman-even-ap} finishes the argument.  For general
$(m,\ell)$, the corresponding identity involves the three dilations
$D_m$, $D_\ell$, and $D_{m+\ell}$ and does not by itself give
$M\cap JM=\emptyset$.  We must instead compare the dilation subsystems
through their maximal equicontinuous factors, deal separately with the
possible two-torsion obstruction when $m$ is even, and then construct the
required sequence by a recursive return-time argument.  Thus the two
proofs share their affine encoding and ultrafilter framework, while their
decisive steps are different.

\subsection{The affine encoding}\label{sec:weighted-affine}

As in \cref{sec3}, let
\[
 \mathcal I=\{(s,r):s\in\N,\ 0\leq r<s\},
 \qquad
 \mathcal X=(\two^{\Z})^{\mathcal I},
\]
where $\mathcal X$ carries the product topology.
For $x\in\mathcal X$, write $x=(x_{s,r})_{(s,r)\in\mathcal I}$, and define
\[
 (Tx)_{s,r}(t)=x_{s,r}(t+1),
 \qquad
 (Jx)_{s,r}(t)=1-x_{s,r}(t).
\]
The affine encoding of $c$ is the point $\mathbf c\in\mathcal X$ given by
\[
 \mathbf c_{s,r}(t)=c(st+r).
\]
The shift $T$ is a homeomorphism and $J$ is a fixed-point-free involution
commuting with $T$.
Put
\[
 X_{\mathbf c}=\overline{\{T^n\mathbf c:n\in\Z\}}.
\]
Then $(X_{\mathbf c},T)$ is the two-sided orbit-closure system generated
by the affine encoding.

For $k\in\N$ and $t\in\Z$, write uniquely $t=kq+j$ with
$q\in\Z$ and $0\leq j<k$, and define
\begin{equation}\label{eq:deltak}
 (\Delta_kx)_{s,r}(t)=x_{ks,sj+r}(q).
\end{equation}
This is well defined because
$0\leq sj+r<ks$, so $(ks,sj+r)\in\mathcal I$.  For example,
\[
 (\Delta_k\mathbf c)_{s,r}(t)
 =c(ksq+sj+r)=c(st+r)=\mathbf c_{s,r}(t).
\]
The other two identities below follow by the same substitution:
\begin{equation}\label{eq:delta-relations}
 \Delta_k\mathbf c=\mathbf c,
 \qquad
 \Delta_kT=T^k\Delta_k,
 \qquad
 \Delta_kJ=J\Delta_k.
\end{equation}

The action of $\bN$ and the dilation maps $D_k$ are those fixed in
\cref{sec2}.  From \eqref{eq:delta-relations} and continuity,
\begin{equation}\label{eq:delta-ultra}
 \Delta_k(px)=(D_k p)(\Delta_kx)
\end{equation}
for $p\in\bN$ and $x\in\mathcal X$.  We shall repeatedly use
\cref{lem:ultrafilter-facts}, in particular its residue decomposition
with the relevant value of $k$, and the omega-limit description in
\cref{lem:omega-ultrafilter}.

\subsection{A key identity for the encoded point}

The following proposition is the weighted counterpart of
\cref{prop:key-identity}.
\begin{prop}\label{prop:key}
For every $p\in\Nstar$,
\begin{equation}\label{eq:key}
 (D_m p+D_\ell p)\mathbf c=J(D_{m+\ell} p)\mathbf c.
\end{equation}
\end{prop}

Before proving it, we need the following selection lemma.

\begin{lemma}\label{lem:weighted-selection}
Let $p\in\Nstar$ and $A\subset\Nzero$. If
\[
 A\in D_m p+D_\ell p
 \qquad\text{and}\qquad
 A\in D_{m+\ell} p,
\]
then there is a strictly increasing sequence $v_1<v_2<\cdots$ such that
\[
 (m+\ell)v_i\in A,
 \qquad
 mv_i+\ell v_j\in A
 \quad(i<j).
\]
The sequence may be chosen above any prescribed bound.
\end{lemma}

\begin{proof}
By the definitions of dilation and ultrafilter addition,
\[
 A\in D_{m+\ell} p
 \quad\Longleftrightarrow\quad
 H:=\{n:(m+\ell)n\in A\}\in p,
\]
and
\[
 A\in D_m p+D_\ell p
 \quad\Longleftrightarrow\quad
 K:=\left\{x:\{y:mx+\ell y\in A\}\in p\right\}\in p.
\]
For $x\in K$, write $Y_x=\{y:mx+\ell y\in A\}\in p$. Choose
$v_1\in H\cap K$ above the prescribed bound. Having chosen
$v_1<\cdots<v_{j-1}$, choose $v_j>v_{j-1}$ in
\[
 H\cap K\cap\bigcap_{i<j}Y_{v_i}.
\]
All these sets belong to the free ultrafilter $p$.
\end{proof}

\begin{proof}[Proof of \cref{prop:key}]
Suppose that \eqref{eq:key} fails. Since every coordinate is binary, there
are $(s,r)\in\mathcal I$, $t\in\Z$, and $\gamma\in\{0,1\}$ such that the two
coordinates
\[
 \bigl((D_m p+D_\ell p)\mathbf c\bigr)_{s,r}(t)
 \quad\text{and}\quad
 \bigl((D_{m+\ell} p)\mathbf c\bigr)_{s,r}(t)
\]
are both equal to $\gamma$. Put
\[
 A=\{n\in\Nzero:c(s(n+t)+r)=\gamma\}.
\]
The definition of the ultrafilter action on the shift shows that the two
coordinate equalities are precisely
$A\in D_m p+D_\ell p$ and $A\in D_{m+\ell} p$. Apply
\cref{lem:weighted-selection}.  Using its last assertion, choose the
$v_i$ above a bound for which all the integers occurring below are
positive, and put
\[
 R=st+r,
 \qquad
 a_i=s(m+\ell)v_i+R.
\]
Then $(a_i)$ is a strictly increasing sequence in $\N$ and
\[
 b((m+\ell)a_i)=c(a_i)=\gamma.
\]
For $i<j$,
\[
 ma_i+\ell a_j
 =(m+\ell)\bigl(s(mv_i+\ell v_j)+R\bigr),
\]
so
\[
 b(ma_i+\ell a_j)
 =c\bigl(s(mv_i+\ell v_j)+R\bigr)=\gamma.
\]
This contradicts the choice of $b$.
\end{proof}

\subsection{Minimal subsystems of \texorpdfstring{$\omega(\mathbf c)$}{the omega-limit set}}

Let
\[
 \Omega=\omega(\mathbf c)
 =\bigcap_{N\geq0}\overline{\{T^n\mathbf c:n\geq N\}}.
\]
Because $T$ is a homeomorphism, $\Omega$ is invariant under both $T$ and
$T^{-1}$. By \cref{lem:omega-ultrafilter}, every point of $\Omega$ has
the form $p\mathbf c$ for some $p\in\Nstar$.

First, we prove that $\Omega$ is $J$-invariant.
\begin{lemma}\label{lem:JOmega}
One has
\begin{equation}\label{eq:JOmega}
 J\Omega=\Omega.
\end{equation}
\end{lemma}

\begin{proof}
Take $x=q\mathbf c\in\Omega$, with $q\in\Nstar$. By
\cref{lem:ultrafilter-facts}(iii), write
$q=\varepsilon+D_{m+\ell} p$, where $0\leq\varepsilon<m+\ell$ and $p\in\Nstar$. Then \cref{prop:key} gives
\[
 \begin{aligned}
 Jx
 &=J(\varepsilon+D_{m+\ell} p)\mathbf c
  =T^\varepsilon J(D_{m+\ell} p)\mathbf c\\
 &=T^\varepsilon(D_m p+D_\ell p)\mathbf c
  =(\varepsilon+D_m p+D_\ell p)\mathbf c.
 \end{aligned}
\]
The ultrafilter on the last line is free because its rightmost summand
$D_\ell p$ is free. Hence $Jx\in\Omega$, proving $J\Omega\subset\Omega$.
Applying $J$ once more gives the reverse inclusion.
\end{proof}

\begin{prop}\label{prop:MJM}
If $M\subset\Omega$ is minimal, then
\[
 M=JM.
\]
\end{prop}

\begin{proof}
By \cref{lem:JOmega}, $JM$ is minimal. Hence either $M=JM$ or
$M\cap JM=\emptyset$. Assume the latter. By
\cref{lem:clopen-domain}, there is a clopen set $H\subset\Omega$ such
that
\[
 M\subset H,
 \qquad
 \Omega=H\sqcup JH.
\]
Applying \cref{lem:asymptotic-binary-coding} to $X_{\mathbf c}$, $\mathbf c$,
$\Omega$, $J$, and $H$, we obtain a binary sequence
$h:\Nzero\to\{0,1\}$ such that $h^{-1}(1)\cap\N$ is thick in $\N$. For every
$p\in\Nstar$, the ultrafilters $D_{m+\ell}p$ and
$D_mp+D_\ell p$ are free. Note that  Hence \cref{prop:key} and the coding
conclusion give
\begin{equation}\label{eq:key-h}
 (D_m p+D_\ell p)h=J(D_{m+\ell} p)h.
\end{equation}

We show that $h$ has no target configuration. Suppose an infinite
$B\subset\N$ and a color $\gamma$ satisfy
\[
 h((m+\ell)x)=h(mx+\ell y)=\gamma
 \qquad(x,y\in B,\ x<y).
\]
Choose a free ultrafilter $p$ containing $B$ and put
$A_\gamma=h^{-1}(\gamma)$. The diagonal condition gives
$A_\gamma\in D_{m+\ell} p$. For each $x\in B$, the set
\[
 \{y:mx+\ell y\in A_\gamma\}
\]
contains $B\cap(x,\infty)$, which belongs to $p$. Hence
\[
 \{x:\{y:mx+\ell y\in A_\gamma\}\in p\}\in p,
\]
so $A_\gamma\in D_m p+D_\ell p$. Taking the zeroth coordinate in
\eqref{eq:key-h} gives $\gamma=1-\gamma$, a contradiction.  Thus $h$ has no target
configuration.

On the other hand, $h^{-1}(1)\cap\N$ is thick. Fix $L\geq1$. There is $n\geq1$
	such that
	$
	\{n,n+1,\ldots,n+(m+l)L\}\subseteq h^{-1}(1)\cap\N$.
	Let $r_L$ be the least integer not smaller than $n$ with $(m+l)|r_L$.  Then
	$
	r_L,r_L+(m+l),\ldots,r_L+(L-1)(m+l)\in h^{-1}(1)\cap\N.
	$
Thus the hypothesis of \cref{thm2-1} holds.  That theorem gives  that $h$ has target configuration, a contradiction. 
This
contradiction proves
$JM=M$. This completes the proof.
\end{proof}

\subsection{Dilation subsystems and maximal equicontinuous factors}\label{sec:weighted-mef}

For a minimal subsystem $M\subset\Omega$ and $k\in\N$, define
\begin{equation}\label{eq:Gamma}
 \Gamma_k(M)=\bigcup_{r=0}^{k-1}T^r\Delta_k(M).
\end{equation}
If $x=p\mathbf c\in M$ with $p\in\Nstar$, then
$\Delta_kx=(D_k p)\mathbf c\in\Omega$; hence
$\Gamma_k(M)\subset\Omega$. The map $\Delta_k$ intertwines $T$ with
$T^k$, so $(\Delta_k(M),T^k)$ is a factor of the minimal system $(M,T)$
and is minimal.  The union in \eqref{eq:Gamma} adds the $k$ possible
phases and is $T$-invariant: the last phase returns to the first because
$T^k\Delta_k(M)=\Delta_k(M)$.  If $z=\Delta_kx$ with $x\in M$, then for
$0\leq r<k$,
\[
 T^{kn+r}z=T^r\Delta_k(T^nx).
\]
As $\{T^nx:n\geq0\}$ is dense in $M$, the $T$-orbit of $z$ is dense in
$\Gamma_k(M)$. Thus $\Gamma_k(M)$ is a minimal $T$-system. By \cref{prop:MJM} and
\eqref{eq:delta-relations}, it is also $J$-invariant.

\begin{lemma}\label{lem:Gamma-equal}
For every minimal $M\subset\Omega$,
\[
 \Gamma_\ell(M)=\Gamma_{m+\ell}(M).
\]
\end{lemma}

\begin{proof}
Take $y=p\mathbf c\in M$ with $p\in\Nstar$. By \cref{prop:key}
and \eqref{eq:delta-ultra},
\[
 J\Delta_{m+\ell}y=(D_m p)(\Delta_\ell y).
\]
The right-hand side belongs to $\Gamma_\ell(M)$ because that compact
$T$-invariant set is also invariant under the $\bN$-action.  The
left-hand side belongs
to $J\Gamma_{m+\ell}(M)$, which equals $\Gamma_{m+\ell}(M)$.  Thus the two
minimal systems intersect and hence are equal.
\end{proof}

Write
\[
 N=\Gamma_\ell(M)=\Gamma_{m+\ell}(M).
\]
Let
\[
 \pi_M:M\to Z_M,
 \qquad
 \pi_N:N\to Z_N
\]
be the maximal equicontinuous factors. Choose origins so that $Z_M$ and $Z_N$ are compact monothetic abelian
groups on which $T$ acts by addition of $g_M$ and $g_N$, respectively. The
cyclic subgroup generated by each of these elements is dense.

Applying \cref{lem:mef-affine}(i) to $J$, there are
$h_M\in Z_M$ and $h_N\in Z_N$ such that
\begin{equation}\label{eq:Jfactor}
 \pi_M(Jx)=\pi_M(x)+h_M,
 \qquad
 \pi_N(Jy)=\pi_N(y)+h_N.
\end{equation}
Since $J^2=\mathrm{id}$,
\begin{equation}\label{eq:twoh}
 2h_M=2h_N=0.
\end{equation}

For $k\in\{\ell,m+\ell\}$, the map
\[
 f_k=\pi_N\circ\Delta_k:M\to Z_N
\]
satisfies $f_k(Tx)=f_k(x)+kg_N$. \cref{lem:mef-affine}(ii) therefore
gives a continuous homomorphism $A_k:Z_M\to Z_N$ and a constant
$\zeta_k\in Z_N$ such that
\begin{equation}\label{eq:fk-affine}
 \pi_N(\Delta_kx)=A_k\pi_M(x)+\zeta_k,
 \qquad
 A_k(g_M)=kg_N.
\end{equation}
For an integer $r$ and a homomorphism $B:Z_M\to Z_N$, write
$rB$ for the homomorphism $z\mapsto rB(z)$. Note that 
\[\ell A_{m+\ell}(rg_M)=\ell(m+\ell)rg_N=(m+\ell)A_\ell(rg_M), \forall r\in\mathbb{N}.\] Since the cyclic subgroup
generated by $g_M$ is dense,
\[
 (m+\ell)A_\ell=\ell A_{m+\ell}.
\]
Since $\gcd(\ell,m+\ell)=\gcd(\ell,m)=1$, choose $u,v\in\Z$ with
$u\ell+v(m+\ell)=1$ and put
$A=uA_\ell+vA_{m+\ell}$. Then
\[
 \begin{aligned}
 \ell A
 &=u\ell A_\ell+v\ell A_{m+\ell}
   =\bigl(u\ell+v(m+\ell)\bigr)A_\ell=A_\ell,\\
 (m+\ell)A
 &=u(m+\ell)A_\ell+v(m+\ell)A_{m+\ell}\\
 &=\bigl(u\ell+v(m+\ell)\bigr)A_{m+\ell}=A_{m+\ell}.
 \end{aligned}
\]
Thus
\begin{equation}\label{eq:AkA}
 A_\ell=\ell A,
 \qquad
 A_{m+\ell}=(m+\ell)A.
\end{equation}
To compare the action of $J$, evaluate
$\pi_N(\Delta_kJx)=\pi_N(J\Delta_kx)$ by
\eqref{eq:fk-affine} and \eqref{eq:Jfactor}.  The constants
$A_k\pi_M(x)+\zeta_k$ cancel, giving
\begin{equation}\label{eq:Ah-relations}
 A_\ell h_M=h_N,
 \qquad
 A_{m+\ell}h_M=h_N.
\end{equation}
Consequently,
\begin{equation}\label{eq:mAh}
 mAh_M=\bigl((m+\ell)-\ell\bigr)Ah_M=0.
\end{equation}

\begin{prop}
\label{prop:Jtrivial}
Assume $\gcd(m,\ell)=1$. Then
\[
 h_N=0.
\]
Equivalently, $J$ acts trivially on the maximal equicontinuous factor of $N$.
\end{prop}

\begin{proof}
If $m$ is odd, then \eqref{eq:twoh} gives
$2Ah_M=0$, while \eqref{eq:mAh} gives $mAh_M=0$. Since
$\gcd(2,m)=1$, it follows that $Ah_M=0$. Equation \eqref{eq:AkA} and 
\eqref{eq:Ah-relations} then gives $h_N=\ell Ah_M=0$.

Assume now that $m$ is even. Since $\gcd(m,\ell)=1$, both $\ell$ and
$m+\ell$ are odd. Suppose, toward a contradiction, that $h_N\neq0$.
The purpose of the next construction is to make this possible nonzero
two-torsion translation visible as a circle phase.  The affine identity
will lift to the circle extension, while the phase relation over $M$ will
separate a minimal subsystem from its color complement.  A clopen coding
of that separation will then contradict the thick-cell lemma.
By Pontryagin duality, continuous characters separate points of compact
abelian groups \cite[Theorem~1.7.2]{Rudin62}, so there is a continuous homomorphism
$\chi:Z_N\to\Torus$ with
$\chi(h_N)\neq0$. Since $2h_N=0$, the element $\chi(h_N)$ is a nonzero
two-torsion point of $\Torus$, and hence
\begin{equation}\label{eq:chi-h}
 \chi(h_N)=\tfrac12.
\end{equation}
Put $\alpha=\chi(g_N)$. For $p\in\beta\Nzero$, define
\[
 \eta(p)=p\text{-}\lim_n n\alpha\in\Torus.
\]
For the rotation $R_\alpha(t)=t+\alpha$, the action of $p$ is translation by
$\eta(p)$. Hence
\[
 (p+q)0=p(q0)=\eta(q)+\eta(p),
\]
which gives the first identity below; the second follows by taking the
$p$-limit of $kn\alpha$:
\begin{equation}\label{eq:eta-hom}
 \eta(p+q)=\eta(p)+\eta(q),
 \qquad
 \eta(D_k p)=k\eta(p).
\end{equation}

Consider
\[
 \widetilde{\mathcal X}=\mathcal X\times\Torus,
 \qquad
 \widetilde T(x,t)=(Tx,t+\alpha),
 \qquad
 \widetilde J(x,t)=(Jx,t),
 \qquad
 \widetilde{\mathbf c}=(\mathbf c,0),
\]
and put $\widetilde\Omega=\omega(\widetilde{\mathbf c})$. For
$p\in\Nstar$,
\[
 p\widetilde{\mathbf c}=(p\mathbf c,\eta(p)).
\]
The affine identity lifts to the extension:
\begin{equation}\label{eq:lifted-key}
 (D_m p+D_\ell p)\widetilde{\mathbf c}
 =\widetilde J(D_{m+\ell} p)\widetilde{\mathbf c}
 \qquad(p\in\Nstar).
\end{equation}
Indeed, the first coordinates agree by \cref{prop:key}; by
\eqref{eq:eta-hom}, both second coordinates are $(m+\ell)\eta(p)$.

The lifted omega-limit set is also invariant under color
complementation:
\begin{equation}\label{eq:JtildeOmega}
 \widetilde J\widetilde\Omega=\widetilde\Omega.
\end{equation}
Take $z=q\widetilde{\mathbf c}\in\widetilde\Omega$ with $q\in\Nstar$.
By \cref{lem:ultrafilter-facts}(iii), write
$q=\varepsilon+D_{m+\ell} p$, where $0\le\varepsilon<m+\ell$ and $p\in\Nstar$. Put
\[
 q'=\varepsilon+D_m p+D_\ell p.
\]
The rightmost summand $D_\ell p$ is free, so $q'\in\Nstar$.
\Cref{prop:key} gives
\[
 q'\mathbf c
 =T^\varepsilon(D_m p+D_\ell p)\mathbf c
 =T^\varepsilon J(D_{m+\ell} p)\mathbf c
 =J(q\mathbf c),
\]
while \eqref{eq:eta-hom} gives
\[
 \eta(q')=\varepsilon\alpha+(m+\ell)\eta(p)=\eta(q).
\]
Thus $q'\widetilde{\mathbf c}=\widetilde Jz$, proving one inclusion in
\eqref{eq:JtildeOmega}; the reverse follows from $\widetilde J^2=\mathrm{id}$.

Let
\[
 \widetilde K_M
 =\{(x,t)\in\widetilde\Omega:x\in M\}.
\]
Thus $\widetilde K_M$ is exactly the inverse image of $M$ under the
first-coordinate projection restricted to $\widetilde\Omega$. In
particular, every point of $\widetilde\Omega$ whose first coordinate lies
in $M$ belongs to $\widetilde K_M$. Since $M$ is closed and $T$-invariant,
$\widetilde K_M$ is compact and $\widetilde T$-invariant.
By \cref{lem:omega-ultrafilter}, the first-coordinate projection maps
$\widetilde\Omega$ onto $\Omega$: if $x=p\mathbf c\in\Omega$, then
$p\widetilde{\mathbf c}=(x,\eta(p))\in\widetilde\Omega$. Hence
$\widetilde K_M$ is nonempty.
Choose a minimal subsystem $\widetilde M\subset\widetilde K_M$. Its first
coordinate projection is a nonempty closed invariant subset of the minimal
system $M$, and hence equals $M$.  Thus $\widetilde M$ is a genuine
minimal lift of $M$ on which the additional coordinate records the phase.

To determine the phases occurring above $M$, use
\eqref{eq:AkA} and
\eqref{eq:fk-affine},
\[
 \ell\bigl(A(g_M)-g_N\bigr)=0,
 \qquad
 (m+\ell)\bigl(A(g_M)-g_N\bigr)=0.
\]
Since $\gcd(\ell,m+\ell)=1$,
\begin{equation}\label{eq:Ag}
 A(g_M)=g_N.
\end{equation}
Define
\[
 \psi=\chi\circ A\circ\pi_M:M\to\Torus.
\]
Equations \eqref{eq:Ag} and \eqref{eq:chi-h} imply
\begin{equation}\label{eq:psi-dynamics}
 \psi(Tx)=\psi(x)+\alpha.
\end{equation}
Moreover, $Ah_M=h_N$: indeed, $\ell Ah_M=h_N$ by
\eqref{eq:Ah-relations}, $2Ah_M=0$, and multiplication by the odd integer
$\ell$ is the identity on two-torsion. Therefore by \eqref{eq:Jfactor} and \eqref{eq:chi-h} 
\begin{equation}\label{eq:psi-J}
 \psi(Jx)=\psi(x)+\tfrac12.
\end{equation}

Put $d_k=\chi(\zeta_k)$ for $k\in\{\ell,m+\ell\}$. Let
$(x,t)\in\widetilde K_M$. By \cref{lem:omega-ultrafilter}, choose $p\in\Nstar$ with
$(x,t)=p\widetilde{\mathbf c}$. Then $x=p\mathbf c$ and $t=\eta(p)$. By
\eqref{eq:delta-ultra} and \cref{prop:key},
\begin{equation}\label{eq:anchored-key}
 (D_m p)(\Delta_\ell x)=J\Delta_{m+\ell}x.
\end{equation}
Indeed, the left side is
$(D_mp+D_\ell p)\mathbf c$, while the right side is
$J(D_{m+\ell}p)\mathbf c$.
For every $y\in N$ and $q\in\beta\Nzero$, equivariance of $\pi_N$ and
continuity of $\chi$ give
\[
 \chi\pi_N(qy)=\chi\pi_N(y)+\eta(q).
\]
Applying this with $q=D_m p$ to the left side of
\eqref{eq:anchored-key}, and then using \eqref{eq:fk-affine}, \eqref{eq:AkA} and \eqref{eq:eta-hom}, gives
\[
 \begin{aligned}
 \chi\pi_N\bigl((D_m p)(\Delta_\ell x)\bigr)
 &=\chi\pi_N(\Delta_\ell x)+\eta(D_m p)\\
 &=\chi\pi_N(\Delta_\ell x)+m\eta(p)\\
 &=\ell\psi(x)+d_\ell+mt.
 \end{aligned}
\]
For the right side, \eqref{eq:chi-h}, \eqref{eq:Jfactor} and \eqref{eq:fk-affine}, \eqref{eq:AkA} give
\[
 \chi\pi_N(J\Delta_{m+\ell}x)
 =\chi\pi_N(\Delta_{m+\ell}x)+\tfrac12
 =(m+\ell)\psi(x)+d_{m+\ell}+\tfrac12.
\]
Equating the two expressions yields
\[
 \ell\psi(x)+d_\ell+mt
 =(m+\ell)\psi(x)+d_{m+\ell}+\tfrac12.
\]
Thus every point of $\widetilde K_M$ satisfies
\begin{equation}\label{eq:finite-root}
 m\bigl(t-\psi(x)\bigr)=\kappa, \text{ where }
 \kappa=d_{m+\ell}-d_\ell+\tfrac12.
\end{equation}
Let
\[
 \mathcal R=\{r\in\Torus:mr=\kappa\}.
\]
This set has exactly $m$ elements. The continuous function
\[
 \rho_0(x,t)=t-\psi(x)
\]
is $\widetilde T$-invariant on $\widetilde K_M$ by
\eqref{eq:psi-dynamics}. Since every orbit in the minimal system
$\widetilde M$ is dense, this continuous invariant function is constant there;
write its value as $r_0$. Equation \eqref{eq:psi-J} shows that it is equal to
$r_0-1/2$ on $\widetilde J\widetilde M$. Since $m$ is even, for every $r\in\mathcal R$ one has
\[
 m(r-\tfrac12)=mr-\tfrac m2=\kappa\quad\text{in }\Torus.
\]
Thus translation by $-1/2$ permutes $\mathcal R$, and it has no fixed point.
If a point belonged to both $\widetilde M$ and
$\widetilde J\widetilde M$, the function $\rho_0$ would take there both
values $r_0$ and $r_0-1/2$, which is impossible. Consequently,
\begin{equation}\label{eq:lifted-disjoint}
 \widetilde M\cap\widetilde J\widetilde M=\emptyset.
\end{equation}

The separation in \eqref{eq:lifted-disjoint} can be encoded by a clopen
$\widetilde J$-fundamental domain containing $\widetilde M$.
Apply \cref{lem:equivariant-circle-extension} with $K=\Omega$, $E=M$,
and the map $\psi$ defined above. The set $M$ is closed and
$J$-invariant by \cref{prop:MJM}, while \eqref{eq:psi-J} is precisely the
required equivariance. The lemma gives an extension
$\bar\psi:\Omega\to\Torus$ of $\psi$ satisfying
\[
 \bar\psi(Jx)=\bar\psi(x)+\tfrac12.
\]
Define
\[
 \rho(x,t)=t-\bar\psi(x)
 \qquad((x,t)\in\widetilde\Omega).
\]
Then
\begin{equation}\label{eq:rho-J}
 \rho(\widetilde J(x,t))=\rho(x,t)-\tfrac12.
\end{equation}
Choose pairwise disjoint open arcs $I_r$ around the finitely many points
$r\in\mathcal R$, with pairwise disjoint closures, so that
\begin{equation}\label{eq:arcs-equivariant}
 I_{r-1/2}=I_r-\tfrac12
 \qquad(r\in\mathcal R).
\end{equation}
Let $I=\bigcup_{r\in\mathcal R}I_r$ and
\[
 B_0=\{(x,t)\in\widetilde\Omega:\rho(x,t)\notin I\}.
\]
This is compact. Its first-coordinate projection $P_0$ is disjoint from $M$:
indeed, every point $(x,t)\in\widetilde\Omega$ with $x\in M$ belongs to
$\widetilde K_M$, and there $\rho(x,t)=t-\psi(x)$ lies in $\mathcal R$ by
\eqref{eq:finite-root}. Equations \eqref{eq:rho-J} and
\eqref{eq:arcs-equivariant} show that $B_0$ is $\widetilde J$-invariant;
hence $P_0$ is $J$-invariant. Choose a clopen neighborhood
$W_0$ of $M$ in $\Omega$ with $W_0\cap P_0=\emptyset$, and put
\[
 W=W_0\cap JW_0.
\]
Since $JM=M$, both $W_0$ and $JW_0$ contain $M$.  Thus $W$ is clopen,
$J$-invariant, contains $M$, and
\begin{equation}\label{eq:no-bad-over-W}
 \{(x,t)\in\widetilde\Omega:x\in W\}\cap B_0=\emptyset.
\end{equation}
For $r\in\mathcal R$, set
\[
 E_r=\{(x,t)\in\widetilde\Omega:x\in W,\ \rho(x,t)\in I_r\}.
\]
By \eqref{eq:no-bad-over-W}, the $E_r$ form a finite partition of the clopen
set above $W$. Each $E_r$ is open. Within the clopen set lying above $W$, its complement
is the union of the other $E_s$; hence each $E_r$ is clopen in
$\widetilde\Omega$. Moreover,
\[
 \widetilde JE_r=E_{r-1/2}.
\]
Choose one representative from each pair
$\{r,r-1/2\}\subset\mathcal R$, choosing $r_0$ from its pair, and let
$H_{\mathrm{in}}$ be the union of the corresponding $E_r$. Then
$H_{\mathrm{in}}$ is clopen.  It contains $\widetilde M$ because
$\rho_0=r_0$ there, and
\[
 \{(x,t)\in\widetilde\Omega:x\in W\}
 =H_{\mathrm{in}}\sqcup\widetilde JH_{\mathrm{in}}.
\]
On the clopen $J$-invariant complement $\Omega\setminus W$, apply
\cref{lem:clopen-domain} with $M=\emptyset$ to obtain a clopen
$D\subset\Omega\setminus W$ with
$\Omega\setminus W=D\sqcup JD$. Put
\[
 H_{\mathrm{out}}
 =\widetilde\Omega\cap(D\times\Torus),
 \qquad
 H=H_{\mathrm{in}}\cup H_{\mathrm{out}}.
\]
Then $H$ is clopen in $\widetilde\Omega$, contains $\widetilde M$, and
\begin{equation}\label{eq:Hfundamental}
 \widetilde\Omega=H\sqcup\widetilde JH.
\end{equation}

Apply \cref{lem:asymptotic-binary-coding} to
$\widetilde{\mathbf c}$, $\widetilde\Omega$, and $H$. We obtain a binary
sequence $h:\Nzero\to\{0,1\}$ such that $h^{-1}(1)\cap\N$ is thick in
$\N$; the long blocks of $1$ come from $\widetilde M\subset H$. For
$p\in\Nstar$, both $D_{m+\ell} p$ and $D_m p+D_\ell p$ are free, so
\eqref{eq:lifted-key}, \eqref{eq:Hfundamental}, and
\cref{lem:asymptotic-binary-coding} give
\begin{equation}\label{eq:key-phase-coded}
 (D_m p+D_\ell p)h=J(D_{m+\ell} p)h.
\end{equation}

The sequence $h$ has no infinite monochromatic configuration of the
target form. Indeed, suppose that $B\subset\N$ is an infinite witness of
color $\gamma$. Choose a free ultrafilter $p$ containing $B$, and put
$A_\gamma=h^{-1}(\gamma)$. Since $(m+\ell)x\in A_\gamma$ for every $x\in B$,
\[
 A_\gamma\in D_{m+\ell} p.
\]
For every $x\in B$, the set
\[
 \{y\in\Nzero:mx+\ell y\in A_\gamma\}
\]
contains the $p$-large tail $B\cap(x,\infty)$, and hence belongs to $p$.
Therefore
\[
 \left\{x:\{y:mx+\ell y\in A_\gamma\}\in p\right\}\in p,
\]
which is precisely $A_\gamma\in D_m p+D_\ell p$. Taking the zeroth
coordinate in \eqref{eq:key-phase-coded} gives
$\gamma=1-\gamma$, a contradiction. Thus $h$ has no target
configuration.

On the other hand, $h^{-1}(1)\cap\N$ is thick. Fix $L\geq1$. There is $n\geq1$
	such that
	$
	\{n,n+1,\ldots,n+(m+l)L\}\subseteq h^{-1}(1)\cap\N$.
	Let $r_L$ be the least integer not smaller than $n$ with $(m+l)|r_L$.  Then
	$
	r_L,r_L+(m+l),\ldots,r_L+(L-1)(m+l)\in h^{-1}(1)\cap\N.
	$
Thus the hypothesis of \cref{thm2-1} holds.  That theorem gives  that $h$ has target configuration, a contradiction.  This
contradiction proves
$h_N=0$. The proof is completed.
\end{proof}

\subsection{Return-time construction and completion of the proof}\label{sec:weighted-return}

\begin{thm}\label{thm:reduced}
Suppose that $\gcd(m,\ell)=1$.  Then every two-coloring of $\N$ admits an
infinite set $B\subseteq\N$ for which \eqref{eq:weighted-target} is
monochromatic.
\end{thm}

\begin{proof}
Continue with the counterexample coloring and affine system fixed in
\cref{sec:weighted-affine}. Choose a minimal subsystem $M\subset\Omega$, let
$N=\Gamma_\ell(M)=\Gamma_{m+\ell}(M)$, and put
\[
 Y_\ell=\Delta_\ell(M)\subset N.
\]
Then $(Y_\ell,T^\ell)$ is minimal, and $Y_\ell$ is a clopen
$T^\ell$-minimal component of $N$ by
\cref{lem:power-components}.  Moreover,
$JY_\ell=Y_\ell$, since $JM=M$ and
$J\Delta_\ell=\Delta_\ell J$.

For $0\leq a<\ell$, put
\[
 \Omega_a=\omega_{T^\ell}(T^a\mathbf c).
\]
Passing to a subsequence of orbit times with a fixed residue modulo
$\ell$ gives
\[
 \omega_T(\mathbf c)=\bigcup_{a=0}^{\ell-1}\Omega_a.
\]
More explicitly, if $T^{n_i}\mathbf c\to x$, pass to a subsequence with
$n_i\equiv a\pmod\ell$ and write
$n_i=a+\ell r_i$, where $r_i\in\Nzero$ and $r_i\to\infty$; then
$x\in\Omega_a$.  The reverse inclusion is immediate.
Choose $x\in M\cap\Omega_a$ for some $a$. Let $M_a$ be the
$T^\ell$-minimal component of $M$ containing $x$. Since $\Omega_a$ is closed
and $T^\ell$-invariant,
\begin{equation}\label{eq:phase-omega}
 M_a\subset\omega_{T^\ell}(T^a\mathbf c).
\end{equation}
Let
\[
 Y_a=\Delta_\ell(M_a).
\]
Since $M_a\subset M$, one has $Y_a\subset Y_\ell$.  The map
$\Delta_\ell$ intertwines $T^\ell$ with $T^{\ell^2}$, so $Y_a$ is
$T^{\ell^2}$-minimal. \cref{lem:power-components}, applied to $(Y_\ell,T^\ell)$, shows
that the $T^{\ell^2}$-minimal components of $Y_\ell$ are clopen. Hence
$Y_a$ is clopen in $Y_\ell$.

For $\gamma\in\{0,1\}$, let
\[
 C_\gamma=\{x\in\mathcal X:x_{1,0}(0)=\gamma\}.
\]
The two clopen sets $M_a\cap\Delta_{m+\ell}^{-1}C_0$ and
$M_a\cap\Delta_{m+\ell}^{-1}C_1$ partition $M_a$. Choose $\gamma$ such that
\[
 E=M_a\cap\Delta_{m+\ell}^{-1}C_\gamma
\]
is nonempty. By \cref{lem:semiopen}, the factor map
\[
 \Delta_\ell:(M_a,T^\ell)\to(Y_a,T^{\ell^2})
\]
is semi-open. Hence $\Delta_\ell(E)$ has nonempty relative interior in
$Y_a$. Choose a nonempty clopen set $O_0^{\rm rel}\subset Y_a$ such that
\begin{equation}\label{eq:O0rel}
 O_0^{\rm rel}\subset\Delta_\ell(E).
\end{equation}
Here we use that $Y_a$, as a subspace of the zero-dimensional shift
$\mathcal X$, has a clopen base.
Since $Y_a$ is clopen in $Y_\ell$, the sets $O_0^{\rm rel}$ and
$Y_\ell\setminus O_0^{\rm rel}$ are disjoint compact subsets of the
zero-dimensional space $\mathcal X$. Hence there is a clopen set
$O_0\subset\mathcal X$ satisfying
\begin{equation}\label{eq:O0ambient}
 O_0\cap Y_\ell=O_0^{\rm rel}.
\end{equation}
Thus $O_0$ is an ambient clopen extension of the relative set on which
the recursion will take place.  Similarly, since $M_a$ is clopen in $M$,
choose a clopen
$U_a\subset\mathcal X$ such that
\begin{equation}\label{eq:Ua}
 U_a\cap M=M_a.
\end{equation}

Define
\begin{equation}\label{eq:V}
 V=Y_\ell\cap T^{-ma}C_\gamma.
\end{equation}
Let $\pi_N:N\to Z_N$ be the maximal equicontinuous factor, let
$R:Z_N\to Z_N$ be the rotation induced by $T$, and put
\[
 Z_\ell=\pi_N(Y_\ell).
\]
As the image of the $T^\ell$-minimal system $Y_\ell$, the set
$Z_\ell$ is $R^\ell$-minimal. \cref{lem:power-components} shows that $Z_\ell$ is a clopen
$R^\ell$-minimal component of $Z_N$.

\begin{lemma}\label{lem:Vfull}
One has
\[
 \pi_N(V)=Z_\ell.
\]
\end{lemma}

\begin{proof}
The inclusion $\pi_N(V)\subset Z_\ell$ follows from $V\subset Y_\ell$.
For the reverse inclusion, take $z\in Z_\ell$ and choose $y\in Y_\ell$
with $\pi_N(y)=z$. If
$T^{ma}y\in C_\gamma$, then $y\in V$. Otherwise
$T^{ma}y\in C_{1-\gamma}$. Since $JY_\ell=Y_\ell$ and $J$ reverses the two
colors, $Jy\in V$. \cref{prop:Jtrivial} gives
$\pi_N(Jy)=\pi_N(y)=z$.
\end{proof}

Set $r_0=0$. We construct recursively a strictly increasing sequence
$0<r_1<r_2<\cdots$ and clopen sets
\[
 O_0\supset O_1\supset O_2\supset\cdots
\]
such that $O_j\cap Y_\ell\neq\emptyset$ and
\begin{equation}\label{eq:Oj}
 O_j=O_{j-1}\cap T^{-\ell mr_j-ma}C_\gamma.
\end{equation}
The inclusion $O_j\subset O_{j-1}$ records all cross conditions imposed
up to stage $j$, while $O_j\cap Y_\ell\neq\emptyset$ guarantees that
the next stage can be continued.
Suppose that $O_{j-1}$ has been constructed. Set
\begin{equation}\label{eq:Wj}
 W_j=\Delta_{m+\ell}^{-1}C_\gamma
 \cap\Delta_\ell^{-1}O_{j-1}
 \cap U_a.
\end{equation}
To see that $W_j\cap M_a\neq\emptyset$, choose
$y\in O_{j-1}\cap Y_\ell$. Since
$O_{j-1}\cap Y_\ell\subset O_0^{\rm rel}\subset\Delta_\ell(E)$,
there is $x\in E$ with $\Delta_\ell x=y$. Then
$x\in M_a\subset U_a$, $\Delta_{m+\ell}x\in C_\gamma$, and
$\Delta_\ell x\in O_{j-1}$, so $x\in W_j\cap M_a$.

Define
\begin{equation}\label{eq:Pj}
 P_j=\{r\in\N:T^{a+\ell r}\mathbf c\in W_j\}.
\end{equation}
Apply \cref{lem:ps-return} to the system
$(\mathcal X,T^\ell)$, the point $T^a\mathbf c$, the minimal subsystem
$M_a\subset\omega_{T^\ell}(T^a\mathbf c)$ from
\eqref{eq:phase-omega}, and the open set $W_j$. Since
$W_j\cap M_a\neq\emptyset$, the set $P_j$ is piecewise syndetic.
Membership in $P_j$ will enforce the new diagonal condition through the
first factor in \eqref{eq:Wj}; the other two factors keep the chosen
orbit in the correct phase and inside $O_{j-1}$.

Next put
\begin{equation}\label{eq:Hj}
 H_j=\left\{r\in\N:
 (O_{j-1}\cap Y_\ell)\cap T^{-\ell mr}V\neq\emptyset
 \right\}.
\end{equation}
Put
\[
 U_j=O_{j-1}\cap Y_\ell
 \qquad\text{and}\qquad
 q=\ell m.
\]
Since the $T^q$-minimal components form a finite clopen partition of
$N$, choose one, denoted by $C_j$, that meets the nonempty open set
$U_j$.
Since $Y_\ell$ is closed and $T^q$-invariant, the nonempty set
$C_j\cap Y_\ell$ is closed and $T^q$-invariant in the minimal system
$(C_j,T^q)$. Hence
\[
 C_j\subset Y_\ell.
\]
By \cref{lem:power-mef},
\[
 C_j=\pi_N^{-1}(\pi_N(C_j)),
\]
and $\pi_N|_{C_j}$ is the maximal equicontinuous factor of
$(C_j,T^q)$. Moreover, \cref{prop:Jtrivial} and the displayed saturation
show that $J(C_j)\subset C_j$; applying $J^2=\mathrm{id}$ gives
\[
 JC_j=C_j.
\]

We next verify that $V\cap C_j$ projects onto this entire factor.
Since $C_j\subset Y_\ell$, one has $\pi_N(C_j)\subset Z_\ell$.
By \cref{lem:Vfull}, for each $z\in\pi_N(C_j)$ there is $v\in V$ with
$\pi_N(v)=z$. The saturation of $C_j$ then forces $v\in C_j$, and hence
\[
 \pi_N(V\cap C_j)=\pi_N(C_j).
\]
Applying \cref{lem:full-fibre-returns} to the minimal system
$(C_j,T^q)$ and the sets $U_j\cap C_j$ and $V\cap C_j$ is legitimate:
they are relatively open in $C_j$, the first is nonempty by the choice of
$C_j$, and the second is nonempty by the preceding full-projection
identity. We obtain that
\[
 H_j'=\left\{r\in\N:
 (U_j\cap C_j)\cap T^{-qr}(V\cap C_j)\neq\emptyset
 \right\}
\]
is thickly syndetic. Since $H_j'\subset H_j$ and supersets of thickly
syndetic sets are thickly syndetic, $H_j$ is thickly syndetic. The tail
\[
 P_j\cap(r_{j-1},\infty)
\]
is still piecewise syndetic, because it differs from $P_j$ by a finite
set. Since every thickly syndetic subset of $\N$ meets every piecewise
syndetic subset of $\N$, choose
\[
 r_j\in H_j\cap P_j\cap(r_{j-1},\infty),
\]
and define $O_j$ by \eqref{eq:Oj}. Since $r_j\in H_j$, there is
$y\in O_{j-1}\cap Y_\ell$ with $T^{\ell mr_j}y\in V$. By the definition
of $V$, this means $T^{\ell mr_j+ma}y\in C_\gamma$, so
$y\in O_j\cap Y_\ell$. Thus $O_j\cap Y_\ell\neq\emptyset$, completing the recursive step.

Put
\[
 n_j=a+\ell r_j.
\]
The sequence $(n_j)$ is strictly increasing and consists of positive
integers. As $r_j\in P_j$, one has $T^{n_j}\mathbf c\in W_j$. Thus
\[
 \Delta_{m+\ell}T^{n_j}\mathbf c\in C_\gamma.
\]
Using \eqref{eq:delta-relations} and
$\Delta_{m+\ell}\mathbf c=\mathbf c$, the point on the left is
$T^{(m+\ell)n_j}\mathbf c$.  Its $(1,0)$-coordinate at $0$ is
$c((m+\ell)n_j)$, and therefore
\begin{equation}\label{eq:diag-final}
 c((m+\ell)n_j)=\gamma.
\end{equation}
Moreover, by \eqref{eq:Wj} and $T^{nj}\mathbf c\in W_j$, we get
\[
 \Delta_\ell T^{n_j}\mathbf c
 =T^{\ell n_j}\mathbf c\in O_{j-1}.
\]
If $i<j$, then $O_{j-1}\subset O_i$, and by \eqref{eq:Oj},
\[
 T^{\ell mr_i+ma}\Delta_\ell T^{n_j}\mathbf c\in C_\gamma.
\]
Again using $\Delta_\ell T^{n_j}=T^{\ell n_j}\Delta_\ell$ and
$\Delta_\ell\mathbf c=\mathbf c$, the $(1,0)$-coordinate of the point
on the left is $c(\ell mr_i+ma+\ell n_j)$.  Since
\[
 \ell mr_i+ma+\ell n_j
 =m(a+\ell r_i)+\ell n_j
 =mn_i+\ell n_j,
\]
we conclude that
\begin{equation}\label{eq:cross-final}
 c(mn_i+\ell n_j)=\gamma
 \qquad(i<j).
\end{equation}
Equations \eqref{eq:diag-final} and \eqref{eq:cross-final} contradict the
assumption that $c$ has no configuration of the form \eqref{eq:c-target}.
\end{proof}

\begin{proof}[Proof of \cref{main2}]
Let $d=\gcd(m,\ell)$, $m=dm_0$, and $\ell=d\ell_0$. Given
$b:\N\to\{0,1\}$, apply \cref{thm:reduced} to the coloring
$\widetilde b(n)=b(dn)$ and the coprime pair $(m_0,\ell_0)$. If
$a_1<a_2<\cdots$ is the resulting sequence and $\gamma$ is its common
color, then
\[
 \widetilde b((m_0+\ell_0)a_i)=b((m+\ell)a_i)=\gamma
\]
and
\[
 \widetilde b(m_0a_i+\ell_0a_j)=b(ma_i+\ell a_j)=\gamma
 \qquad(i<j).
\]
Thus the same sequence is a witness for $b$ with the original coefficients.
\end{proof}

\section{Counterexamples}\label{sec6}
The constructions in this section are adapted from \cite{Neil79,LianXiao24}.
We include complete proofs because the precise placement of the coefficients
and shifts will be used below.
\subsection{The three-color obstruction}
We first show that the two-color hypothesis in our results is essential.

\begin{prop}\label{prop6-1}
Let $m,\ell\in\N$. There is a $3$-coloring of $\N$ such that, for every infinite $B\subseteq\N$ and every $t\in\Z$, the set
\[
 \{(m+\ell)x:x\in B\}
 \cup
 \{mx+\ell y+t:x,y\in B,\ x<y\},
\]
whenever contained in $\N$, is not monochromatic.
\end{prop}

\begin{proof}
Put $\lambda=(m+\ell)/\ell$ and $w=\lambda^2$. Thus $w>1$ and $\lambda=w^{1/2}$. Define $\phi:(0,\infty)\to\{0,1,2\}$ by
\[
 \phi(z)=r
 \quad\Longleftrightarrow\quad
 \{\log_w z\}\in\left[\frac r3,\frac{r+1}{3}\right),
 \qquad r\in\{0,1,2\},
\]
where $\{u\}=u-\lfloor u\rfloor$. Restrict $\phi$ to $\N$.

Suppose that an infinite set $B\subseteq\N$, an integer $t$, and a color $\gamma$ make the configuration in the statement monochromatic. Fix $a\in B$. Since every infinite subset of $\N$ is unbounded, there are elements $b\in B$, with $b>a$, as large as desired. For such $b$, put
\[
 u_b=ma+\ell b+t,
 \qquad
 v_b=(m+\ell)b.
\]

Note that
\[
\lim_{b\in B\to\infty}\frac{v_b}{u_b}
=\lim_{b\to\infty}\frac{(m+\ell)b}{ma+\ell b+t}
=\frac{m+\ell}{\ell}
=\lambda=w^{1/2}.
\]
Thus, for all sufficiently large $b\in B$, we have
\[
 \delta_b:=\log_w(v_b)-\log_w(u_b)
 =\log_w\left(\frac{v_b}{u_b}\right)\in(1/3,2/3).
\]
Since $0<\delta_b<1$, either
$\lfloor\log_w(v_b)\rfloor=\lfloor\log_w(u_b)\rfloor$, in which case
\[
 \bigl|\{\log_w(v_b)\}-\{\log_w(u_b)\}\bigr|=\delta_b,
\]
or $\lfloor\log_w(v_b)\rfloor=\lfloor\log_w(u_b)\rfloor+1$, in which
case this absolute difference is $1-\delta_b$. In either case,
\[
 \bigl|\{\log_w(v_b)\}-\{\log_w(u_b)\}\bigr|\in(1/3,2/3).
\]
The two fractional parts therefore cannot belong to the same one of
the three intervals that define $\phi$. Hence
$\phi(u_b)\neq\phi(v_b)$, a contradiction. This completes the proof.
\end{proof}

\subsection{The asymmetric obstruction}\label{sec7-2}
We next give the asymmetric counterexample mentioned after
\cref{main2}.

\begin{prop}\label{prop6-2}
For any $\ell,m\in \N$ with $\ell\neq m$, there is a $2$-coloring of
$\N$ such that, for every infinite $B\subset\N$ and all
$t_1,t_2\in\mathbb Z$, the set
\[
\{mx+\ell y+t_1:x,y\in B,\ x<y\}
\cup
\{mx+\ell y+t_2:x,y\in B,\ x>y\},
\]
whenever contained in $\N$, is not monochromatic.
\end{prop}
\begin{proof}
Exchanging $m$ and $\ell$ and simultaneously interchanging the two ordered
pieces if necessary, it suffices to consider $\ell<m$. Put
$r=m/\ell>1$ and color $(0,\infty)$ as follows:
\begin{itemize}
	\item color $1$: $[r^{2n},r^{2n+1})$ for all $n\in\Nzero$;
	\item color $2$: $[r^{2n+1},r^{2(n+1)})$ for all $n\in\Nzero$;
	\item color $1$: otherwise (that is, $(0,1)$).
\end{itemize}
Restrict this coloring to $\N$ and denote it by $\phi$.

\medskip
Assume that there exist an infinite $B\subseteq \N$ and
$t_1,t_2\in\mathbb Z$ such that $\phi$ takes a constant value
$c\in\{1,2\}$ on
\[
\{mx+\ell y+t_1:x,y\in B,x<y\}
\cup\{\ell x+my+t_2:x,y\in B,x<y\}.
\]
Here the second family is the $x>y$ family in the statement after the two
variables have been interchanged.

We choose $n_1,n_2,k$ successively. Since
\[
 rm-\ell=\frac{m^2-\ell^2}{\ell}>0,
\]
we may first choose $n_1\in B$ so large that $mn_1+t_1>0$ and
\begin{equation}\label{eq:asym-n1}
 \ell n_1+t_2<r(mn_1+t_1).
\end{equation}
Next choose $n_2\in B$, $n_2>n_1$, so large that
\begin{equation}\label{eq:asym-n2}
 r(mn_1+t_1)<\ell n_2+t_2.
\end{equation}
Finally, choose $k\in B$, $k>n_2$, sufficiently large. For this choice
the following conditions hold:
\begin{itemize}
\item[(1)] $\ell n_2+mk+t_2,\ell n_1+mk+t_2,\ell k+mn_1+t_1>1$.
\item[(2)] $r(mn_{1}+t_1)<\ell n_2+t_2$.
\item[(3)] $1<(\ell n_2+mk+t_2)/(\ell n_1+mk+t_2)<r,1<(\ell n_1+mk+t_2)/(\ell k+mn_1+t_1)<r$.
\end{itemize}
Indeed, (1) holds for all sufficiently large $k$, while (2) is
\eqref{eq:asym-n2}. Moreover,
\[
 \frac{\ell n_2+mk+t_2}{\ell n_1+mk+t_2}\longrightarrow1
 \qquad\text{as }k\to\infty.
\]
The quotient is greater than $1$ because $n_2>n_1$, and hence it is
less than $r$ for all sufficiently large $k$. Similarly,
\[
 \frac{\ell n_1+mk+t_2}{\ell k+mn_1+t_1}\longrightarrow
 \frac m\ell=r.
\]
The quotient is less than $r$ by \eqref{eq:asym-n1}, and it is greater
than $1$ for all sufficiently large $k$, since $m>\ell$. This proves
(3), and the required $k$ exists because $B$ is unbounded.
Next, we introduce a claim.
\begin{cl}
	If $x,y>1$, $1<x/y<r$, and $\phi(x)=\phi(y)$, then $x,y\in [r^d,r^{d+1})$ for some $d\in \Nzero$.
\end{cl}
\begin{proof}
	Suppose not. At this point, we take $d\in \Nzero$ such that $x\in [r^d,r^{d+1})$ and $y\notin [r^d,r^{d+1})$. By $1<x/y<r$, $r^{d-1}\le y<r^{d}$. This contradicts $\phi(x)=\phi(y)$. Therefore, $x,y\in [r^d,r^{d+1})$ for some $d\in \Nzero$.
\end{proof}
Now, assume that $(\ell n_1+mk+t_2)\in [r^d,r^{d+1})$ for some $d\in \Nzero$. By the above claim, $(\ell n_2+mk+t_2),(\ell k+mn_1+t_1)\in [r^d,r^{d+1})$. Then $r^{d+1}>(\ell n_2+mk+t_2)$ and $(\ell k+mn_1+t_1)r\ge r^{d+1}$. This contradicts (2). This completes the proof.
\end{proof}
\begin{rem}
	When $\ell=1$, this is Hindman's coloring from
	\cite[proof of Theorem~2.11]{Neil79}; the construction above extends it
	to general $\ell$.
\end{rem}

\subsection{The three-fold version of Owings's question}\label{sec6-3}
If every $2$-coloring of $\N$ admitted an infinite $B\subseteq\N$ with
monochromatic $B+B+B$, then
\[
\{2x+y:x,y\in B,x<y\}\cup\{x+2y:x,y\in B,x<y\}
\]
would be monochromatic, since both displayed families are subsets of
$B+B+B$. Apply \cref{prop6-2} with $(m,\ell)=(2,1)$ and
$t_1=t_2=0$. The coloring constructed there admits no such configuration.
Consequently, it has no infinite $B$ for which $B+B+B$ is monochromatic.

\appendix

\section{Proof of the admissible weighted theorem}\label{appA}

In this appendix we give a proof of \cref{thm2-1}. The idea of the proof is similar to the proof of \cite[Corollary~2.10]{Neil79}.

Fix $m,\ell\in\N$. For $K,M\in\N$, write
\[
Q(K,M):=\{(m+\ell)(K+s):0\leq s<M\}.
\]

\begin{lemma}\label{lem:local}
Let $c:\N\to\{0,1\}$, fix $j\in\{0,1\}$, and put $j'=1-j$. Suppose that there is no infinite set $B\subseteq\N$ such that
\[
c((m+\ell)x)=c(mx+\ell y)=j\qquad(x,y\in B,\ x<y).
\]
Then there is $C_j\in\N$ with the following property. If $K,M,x\in\N$ satisfy $|Q(K,M)\cap c^{-1}(j')|\leq1$, $x>C_j$, and
\[
(m+\ell)K<\ell x<(m+\ell)(K+M)-C_j,
\]
then $c((m+\ell)x)=j'$.
\end{lemma}

\begin{proof}
Suppose otherwise. For every $C\in\N$, choose $K,M,x\in\N$ satisfying the three conditions with $C$ in place of $C_j$, but with $c((m+\ell)x)=j$.

We recursively construct triples $(K_r,M_r,x_r)$. Having chosen $x_1<\cdots<x_r$, take $C>\max_{t\leq r}(m+\ell)x_t$ and choose $(K_{r+1},M_{r+1},x_{r+1})$ as above. Then $x_{r+1}>C>x_r$, and for every $t\leq r$,
\[
(m+\ell)K_{r+1}<mx_t+\ell x_{r+1}<(m+\ell)(K_{r+1}+M_{r+1}).
\]
Indeed, the lower bound follows from $\ell x_{r+1}>(m+\ell)K_{r+1}$, while the upper bound follows from $mx_t<(m+\ell)x_t<C$.

Pass to a subsequence on which all $x_r$ are congruent modulo $m+\ell$. For $t<r$, the integer $mx_t+\ell x_r$ is divisible by $m+\ell$ and lies in the preceding open interval; hence it belongs to $Q(K_r,M_r)$. Apply the infinite Ramsey theorem \cite[Chapter~1, Theorem~5]{GrahamRothchildSpencer90} to the coloring of pairs $t<r$ by $c(mx_t+\ell x_r)$, and pass to a further subsequence on which all cross terms have one color. This color cannot be $j$, since every diagonal term $c((m+\ell)x_r)$ equals $j$. Thus all cross terms have color $j'$. If $r_1<r_2<s$, the two distinct numbers $mx_{r_1}+\ell x_s$ and $mx_{r_2}+\ell x_s$ lie in $Q(K_s,M_s)\cap c^{-1}(j')$, contradicting its cardinality bound. This completes the proof.
\end{proof}

\begin{prop}\label{prop:thick}
Let $c:\N\to\{0,1\}$. If one color class is thick, then there is an infinite set $B\subseteq\N$ such that
\[
(m+\ell)B\cup\{mx+\ell y:x,y\in B,\ x<y\}
\]
is monochromatic.
\end{prop}

\begin{proof}
After interchanging the colors, assume that $c^{-1}(0)$ is thick. Suppose that neither color contains the required configuration. Apply \cref{lem:local} to the two target colors, and denote the resulting constants by $C_0$ and $C_1$.

Choose $L\in\N$ with $mL>C_1$. Put $P(a):=Q(a,L)$, and call $a$ good if $|P(a)\cap c^{-1}(0)|\leq1$. Every sufficiently long integer interval contains as many consecutive multiples of $m+\ell$ as prescribed. Using thickness, and choosing each new interval beyond the preceding ones, choose blocks $Q_n:=Q(K_n,M_n)\subseteq c^{-1}(0)$ such that $M_n\to\infty$ and
\[
K_{n+1}>L+2+\frac{(m+\ell)K_n}{\ell}.
\]

Let $X_n$ be the least positive integer for which $\ell X_n>(m+\ell)K_n$. Then
\[
(m+\ell)K_n<\ell X_n\leq(m+\ell)K_n+\ell.
\]
For all sufficiently large $n$ and every $0\leq s<L$, we therefore have $X_n+s>C_0$ and
\[
(m+\ell)K_n<\ell(X_n+s)\leq(m+\ell)K_n+\ell L<(m+\ell)(K_n+M_n)-C_0.
\]
Since $Q_n$ contains no point of color $1$, \cref{lem:local}, with target color $0$, gives $c((m+\ell)(X_n+s))=1$ for $0\leq s<L$. Thus $X_n$ is good. Moreover, $X_n\leq (m+\ell)K_n/\ell+1<K_{n+1}$.

Discard finitely many terms. For each $n$, let $a_n$ be the largest good integer below $K_{n+1}$. Since $X_{n+1}>K_{n+1}$, we have
\[
a_n<K_{n+1}<X_{n+1}\leq a_{n+1},
\]
so $(a_n)$ is strictly increasing.

Let $z_n$ be the least positive integer for which $\ell z_n>(m+\ell)a_n$. Then $\ell z_n\leq(m+\ell)a_n+\ell$. For all sufficiently large $n$ and every $0\leq s<L$,
\[
(m+\ell)a_n<\ell(z_n+s)\leq(m+\ell)a_n+\ell L<(m+\ell)(a_n+L)-C_1;
\]
the last inequality is precisely $mL>C_1$. Since $a_n$ is good, \cref{lem:local}, with target color $1$, gives
\[
c((m+\ell)(z_n+s))=0\qquad(0\leq s<L).
\]
After discarding finitely many further terms, assume that this holds for every $n$. The sequence $(z_n)$ is strictly increasing, because
\[
\ell z_{n+1}>(m+\ell)a_{n+1}\geq(m+\ell)(a_n+1)>(m+\ell)a_n+\ell\geq\ell z_n.
\]

Pass to an infinite set of indices on which the $z_n$ are congruent modulo $m+\ell$. Color each pair $r<t$ by the $L$-tuple
\[
\bigl(c(m(z_r+s)+\ell(z_t+s))\bigr)_{0\leq s<L}.
\]
The infinite Ramsey theorem \cite[Chapter~1, Theorem~5]{GrahamRothchildSpencer90} gives an infinite index set $I$ on which this tuple is constant. If one coordinate of the constant tuple were $0$, say the coordinate $s$, then $\{z_n+s:n\in I\}$ would give the required configuration in color $0$. Hence
\[
c(m(z_r+s)+\ell(z_t+s))=1
\quad(0\leq s<L,\ r<t,\ r,t\in I).
\]

Let $n_0=\min I$ and put $z_*=z_{n_0}$. For $n\in I$ with $n>n_0$, define
\[
\alpha_n:=\frac{mz_*+\ell z_n}{m+\ell}.
\]
This is an integer because $z_n\equiv z_*\pmod{m+\ell}$. For $0\leq s<L$,
\[
(m+\ell)(\alpha_n+s)=m(z_*+s)+\ell(z_n+s),
\]
so $P(\alpha_n)$ is contained in $c^{-1}(1)$; in particular, $\alpha_n$ is good.

Since $\ell z_n>(m+\ell)a_n$, we have $\alpha_n>a_n$. The maximality of $a_n$ among the good integers below $K_{n+1}$ therefore gives $\alpha_n\geq K_{n+1}$. On the other hand,
\[
\alpha_n\leq a_n+\frac{mz_*+\ell}{m+\ell}<K_{n+1}+D
\]
for some fixed integer $D$. Choose $n\in I$ so large that $M_{n+1}>D+L$. Then $K_{n+1}\leq\alpha_n$ and $\alpha_n+L-1<K_{n+1}+M_{n+1}$, so $P(\alpha_n)\subseteq Q_{n+1}$. This is impossible, since $P(\alpha_n)$ has color $1$ whereas $Q_{n+1}$ has color $0$. This completes the proof.
\end{proof}

Now we can give a proof for \cref{thm2-1}.

\begin{proof}[Proof of \cref{thm2-1}]
Let $c:\N\to\{0,1\}$ be $(m,\ell)$-admissible. Thus there are a color $i$ and $d\in(m+\ell)\N$ such that, for every $R\in\N$, one can find $x_R\in\N$ with
\[
c((m+\ell)x_R+kd)=i\qquad(0\leq k\leq R).
\]
Write $d=(m+\ell)q$. Some residue class $\rho$ modulo $q$ contains $x_R$ for an unbounded set of values of $R$. For these values write $x_R=\rho+qu_R$, where $u_R\in\N_0$, and define
\[
\widetilde c(n):=c((m+\ell)(\rho+qn)),\qquad n\in\N.
\]
Then $\widetilde c(u_R+k)=i$ for $0\leq k\leq R$, apart from the unavailable term $k=0$ when $u_R=0$. It follows that $\widetilde c^{-1}(i)$ is thick.

Apply \cref{prop:thick} to $\widetilde c$. There are a color $\gamma$ and a strictly increasing sequence $k_1<k_2<\cdots$ such that $\widetilde c((m+\ell)k_n)=\gamma$ for every $n$, and $\widetilde c(mk_r+\ell k_t)=\gamma$ whenever $r<t$. Put
\[
b_n:=\rho+q(m+\ell)k_n.
\]
Then $(b_n)$ is strictly increasing and
\[
c((m+\ell)b_n)=\widetilde c((m+\ell)k_n)=\gamma.
\]
For $r<t$,
\[
mb_r+\ell b_t=(m+\ell)\bigl(\rho+q(mk_r+\ell k_t)\bigr),
\]
and hence $c(mb_r+\ell b_t)=\widetilde c(mk_r+\ell k_t)=\gamma$. Thus $B=\{b_n:n\in\N\}$ has the required property. This completes the proof.
\end{proof}

\end{document}